\documentclass[reqno]{amsart}

\usepackage[margin=1.4in]{geometry}
\usepackage{amssymb, amsmath, amsthm, amsfonts, amscd}
\usepackage{mathrsfs,mathtools}
\usepackage[shortlabels]{enumitem}
\usepackage{tikz, tikz-cd}
\usepackage[all]{xy}
\usepackage[new]{old-arrows}
\usepackage{pst-node}
\usepackage{varwidth}
\usepackage{graphicx}
\usepackage{subfig}
\usepackage{tabularx, booktabs}
\usepackage{multirow, multicol}
\usepackage{xcolor}
\usepackage{mathrsfs}
\usepackage{yfonts, euscript}
\usepackage{csquotes, dirtytalk}
\usepackage[pagewise]{lineno} 
\usepackage[colorlinks=true, allcolors=blue]{hyperref}
\hypersetup{
    pdftitle={Automorphisms of Twisted Chevalley Groups of type 2D-l (l >= 4) over Local Rings},
    pdfauthor={Elena I. Bunina and Deep H. Makadiya},
    pdfsubject={Mathematics Subject Classification 2020: 20G35},
    pdfkeywords={Twisted Chevalley groups, Automorphisms, Isomorphisms},
    bookmarksdepth=2
}

\usetikzlibrary{arrows, arrows.meta}

\newtheorem{thm}{{\bf Theorem}}[section]
\newtheorem{lemma}[thm]{{\bf Lemma}}
\newtheorem{prop}[thm]{{\bf Proposition}}

\newtheorem*{rmk}{{\bf Remark}}

\newcommand{\RNum}[1]{\uppercase\expandafter{\romannumeral #1\relax}}

\newcommand{\GL}{\,\mathrm{GL}\,}
\newcommand{\SL}{\,\mathrm{SL}\,}

\newcommand{\PGL}{\,\mathrm{PGL}\,}

\newcommand{\PSO}{\,\mathrm{PSO}\,}
\newcommand{\SO}{\,\mathrm{SO}\,}

\newcommand{\Rad}{\,\mathrm{Rad}\,}

\begin{document}


\title[Automorphisms of Twisted Chevalley Groups of type ${}^2 D_\ell$]{Automorphisms of Twisted Chevalley Groups \\ of type ${}^2 D_\ell \ (\ell \geq 4)$ over Local Rings}


\author{Elena I. Bunina}
\address{Department of Mathematics,
    Bar-Ilan University, Ramat Gan, Israel}
\email{helenbunina@gmail.com}
\thanks{Both authors: Department of Mathematics, Bar-Ilan University,
Ramat Gan, Israel.}
\thanks{Corresponding author: Elena I. Bunina,
\href{mailto:helenbunina@gmail.com}{helenbunina@gmail.com}.}
\thanks{ORCID iDs: Elena I. Bunina,
\href{https://orcid.org/0009-0004-2311-7778}{0009-0004-2311-7778};
Deep H. Makadiya,
\href{https://orcid.org/0009-0000-8410-471X}{0009-0000-8410-471X}}

\author{Deep H. Makadiya}
\address{Department of Mathematics,
    Bar-Ilan University, Ramat Gan, Israel}
\email{deepmakadia25.dm@gmail.com}


\subjclass[2020]{20G35}
\keywords{Twisted Chevalley groups, Automorphisms, Isomorphisms}


\begin{abstract}
    We prove that all automorphisms of twisted Chevalley groups of type ${}^2D_\ell$, $\ell\ge4$, and their elementary subgroups over commutative local rings with $1/2$ are standard. 
    More precisely, for elementary groups, every automorphism is a composition of strictly inner, diagonal, and ring automorphisms; for full groups, a central factor may also occur. In the adjoint case, diagonal automorphisms are strictly inner and the centre is trivial, so only strictly inner and ring factors are needed for both the elementary and full groups. 
\end{abstract}

\maketitle


\section{Introduction}\label{sec:intro}

This is the third paper in our series devoted to the study of automorphisms and isomorphisms of twisted Chevalley groups over commutative rings.
We prove that all automorphisms of twisted Chevalley groups of type ${}^2D_\ell$, $\ell \geq 4$, and their elementary subgroups over commutative local rings containing $1/2$ are standard.

\medskip
\noindent \textbf{Motivation.}
The study of automorphisms and isomorphisms of Chevalley groups over commutative rings has been an active area of research for several decades.
The classical results of Steinberg~\cite{RS1} and Humphreys~\cite{JH0} describe the automorphisms of Chevalley groups over fields.
The groundwork for extending these results to commutative rings were developed by Abe~\cite{EA2,EA3,EA4}, Abe and Suzuki~\cite{EA&KS}, Taddei~\cite{GT}, Vaserstein~\cite{LV}, Vavilov~\cite{NV1}, and others.
Building on these developments, significant progress has been made towards the classification of automorphisms of untwisted Chevalley groups
over commutative rings.
In particular, important contributions were made by Abe~\cite{EA0}, Klyachko~\cite{AK}, and Bunina and her coauthors, who established automorphism classifications over various classes of commutative rings; see, for example, \cite{EB24:final, EB12:main, EB07:first, EB08:b2andg2, EB09:alwithhalfele, EB10:alwithhalf, EB10:alwithouthalf, EB10:blwithhalf, EB10:f4withhalf, EB&PV14:g2withouthlfnor, EB&PV14:g2withouthalf, EB&MV24:g2with1/3}. 

In contrast, the study of automorphisms of twisted Chevalley groups over commutative rings remains considerably less developed, although the corresponding groups over fields are classical objects and their automorphisms are known to be standard in the sense of Steinberg~\cite{RS1}.
Nevertheless, the structural theory of twisted Chevalley groups over commutative rings, including the description of their normal subgroups and other fundamental properties, is now well established through the pioneering work of Abe~\cite{EA1} and Suzuki~\cite{KS1, KS2, KS3}, together with more recent contributions by Garge and Makadiya~\cite{SG&DM1, SG&DM2, SG&DM3}.

Motivated by these developments, we initiated a systematic study of automorphisms and isomorphisms of twisted Chevalley groups over commutative rings through a series of papers.

\medskip
\noindent\textbf{The series and the role of the orthogonal case.}
The first paper in this series~\cite{EB&DM1} deals with the unitary family ${}^2A_\ell$, $\ell \geq 5$, over commutative local rings and establishes that all automorphisms are standard.
The second paper~\cite{EB&DM2} considers the remaining cases ${}^2A_3$ and ${}^2A_4$, where certain elementary subgroups admit nonstandard automorphisms.
These results demonstrate that, unlike the situation over fields, the classification of automorphisms of twisted Chevalley groups over commutative rings may involve additional phenomena, particularly in low-rank cases.

In the present paper, we extend our investigation from the unitary family to the twisted orthogonal family ${}^2D_\ell$, $\ell \geq 4$.
We prove that all automorphisms in this family are standard, including those of the smallest-rank case ${}^2D_4$.
The remaining cases, namely the exceptional type ${}^2E_6$ and the triality type ${}^3D_4$, will be considered in subsequent papers in the series.
Our ultimate goal is to extend these classifications from local rings to arbitrary commutative rings, under suitable invertibility assumptions for each type.
This extension requires more than combining the individual local results: one must establish the compatibility of the local automorphisms and determine how they can be assembled into global automorphisms.
The present paper thus provides the orthogonal component of the local classification needed for this broader program.

We also mention the related work of Gvozdevsky~\cite[Theorem~3.1 and Corollary~14.2]{PG:isotropic}, which establishes results on abstract isomorphisms of elementary subgroups, and more generally of subgroups containing them, for adjoint absolutely simple group schemes satisfying suitable isotropy conditions over commutative rings.
When $R$ is quadratic \'etale over its fixed subring
$R_\theta$, our construction of twisted Chevalley groups
can be interpreted in the usual way through descent
as a reductive group over $R_\theta$.
However, our results do not require the extension
$R/R_\theta$ to be quadratic \'etale and therefore
also apply to non-\'etale local pairs, including the
example above where the induced involution on the
residue field is trivial.
Furthermore, our classification is formulated directly
in terms of the ring with involution $(R,\theta)$
and applies to all $\rho$-invariant weight lattices,
covering both elementary subgroups and full twisted
Chevalley groups.

\medskip
\noindent\textbf{Main results.}
Let $R$ be a commutative local ring with $2\in R^\times$, let $\theta$ be a ring automorphism of order $2$, and let $\rho$ be a diagram automorphism of $D_\ell$ of order $2$.
Consider a representation $\pi$ of the Lie algebra $\mathcal{L}(D_{\ell}, \mathbb{C})$ whose weight lattice $\Lambda_\pi$ is invariant under $\rho$.
The graph automorphism $\rho$ and the ring automorphism $\theta$ commute and hence define an involution $\sigma=\rho\circ\theta$ on $G_\pi(D_\ell,R)$.
We denote the corresponding twisted Chevalley group by
\[
    G_{\pi,\sigma}(D_\ell,R)=G_\pi(D_\ell,R)^\sigma
\]
and its elementary subgroup by $E'_{\pi,\sigma}(D_\ell,R)$.
The assumption that $\Lambda_\pi$ is $\rho$-invariant ensures that the diagram automorphism is well-defined for the chosen representation.
Throughout, we consider abstract group automorphisms, without assuming that they preserve root subgroups or are compatible with a particular matrix realization.

Our main result concerns the adjoint elementary subgroup (Theorem~\ref{MT_local1}).
We prove that every automorphism of $E'_{\mathrm{ad},\sigma}(D_\ell,R)$ can be written as
\[
    \varphi=i_g\circ\mu,
    \qquad g\in G_{\mathrm{ad},\sigma}(D_\ell,R),
    \qquad \mu\in\operatorname{Aut}(R),
\]
where $\mu$ commutes with $\theta$ and $i_g$ denotes conjugation by $g$.
Here, the ring automorphism $\mu$ acts naturally on the root elements by
\[
    x_{[\alpha]}(t)\longmapsto x_{[\alpha]}(\mu(t)).
\]
In particular, the conjugating element $g$ belongs to the full adjoint twisted Chevalley group and need not lie in its elementary subgroup.
Thus, up to conjugation by an element of the full group, every automorphism is induced by an automorphism of the ring with involution $(R,\theta)$.

For an arbitrary $\rho$-invariant weight lattice, we obtain the corresponding classifications for both the elementary subgroup and the full twisted Chevalley group.
More precisely, we prove that every automorphism of $E'_{\pi,\sigma}(D_\ell,R)$ is a composition of strictly inner, diagonal, and ring automorphisms (Theorem~\ref{MT_local2}), while every automorphism of $G_{\pi,\sigma}(D_\ell,R)$ is a composition of strictly inner, diagonal, central, and ring automorphisms (Theorem~\ref{MT_local3}).
As in~\cite[Section~2]{EB&DM1}, an automorphism is called strictly inner if it is induced by conjugation by an element of the corresponding full twisted Chevalley group.
The distinction between strictly inner and diagonal automorphisms arises because an element of the adjoint group need not lift to an element of the group associated with the chosen weight lattice over $R$.
A central automorphism induces the identity on the quotient by the center.
Moreover, no separate graph automorphism is needed in these classifications, since the action of $\rho$ on the $\sigma$-fixed group coincides with that of $\theta$ and is therefore already included among the ring automorphisms.

All these results hold uniformly for $\ell\geq4$.
Although the smallest-rank case ${}^2D_4$ requires an additional argument involving short-root elements, the resulting classification remains the same.
In particular, unlike certain low-rank unitary elementary groups studied in~\cite{EB&DM2}, the twisted orthogonal groups considered here admit no nonstandard automorphisms.

\medskip
\noindent\textbf{Main ideas of the proof.}
We work with the projective realization arising from the natural $2\ell$-dimensional orthogonal representation.
Using the normal subgroup structure established in~\cite{SG&DM1}, we show that the congruence subgroup corresponding to the Jacobson radical is characteristic, which allows us to reduce the problem to the residue field.
The classification over fields then provides an initial decomposition into inner and ring automorphisms.
After lifting and removing the inner factor, we obtain an automorphism whose reduction is induced by a field automorphism $f$ commuting with $\theta_k$.
Importantly, we do not assume that $f$ extends to an automorphism of $R$; instead, the required ring automorphism is constructed later from the action on root elements.

The next step is to normalize the given automorphism.
Using a compatible homomorphic lift to the natural elementary group (Lemma~\ref{lemma:compatible-representatives}), we can work with exact matrix identities rather than identities modulo scalar matrices.
We then construct successive changes of basis, each congruent to the identity modulo $J$, so that the resulting automorphism fixes the elements
\[
    h_{[\alpha]}(-1),\qquad
    w_{[\alpha]}(1),\qquad
    x_{[\alpha]}(1).
\]
The commuting involutions provide the necessary decomposition of the underlying module, while the Weyl group relations determine the compatibility of the resulting blocks.
A suitable collection of root relations then determines the matrices corresponding to root elements with parameter $1$.
The calculations for long roots apply uniformly to all $\ell\geq4$, whereas those for short roots require an additional argument in the case $\ell=4$.

However, the changes of basis used in this normalization need not belong to the original orthogonal group.
We therefore modify their product to obtain a conjugating matrix that induces an inner automorphism of the twisted orthogonal group, without affecting the normalization already achieved.
First, by averaging the exact intertwining identities, we arrange that the conjugating matrix $C$ is fixed by $\sigma$.
Next, writing $S$ for the matrix of the split orthogonal form and $Q$ for the matrix realizing the diagram automorphism $\rho$, an explicit calculation of the centralizer yields
\[
    S^{-1}C^tSC=aI+bQ,
    \qquad a,b\in R_\theta,
\]
where $R_\theta$ denotes the fixed subring of $\theta$.
Using the same centralizer, we construct a correction of the form $C(I+dQ)$ that preserves all the previously normalized elements $x_{[\alpha]}(1)$ and is a proper
orthogonal similarity (Proposition~\ref{prop:normalization_of_C}).
The necessary square root in the construction of $d$ is obtained from a determinant identity involving $C$.
Consequently, the entire correction is carried out over $R$, without extending the ring.

Once the normalization is complete, it remains to determine the action of the automorphism on root elements with arbitrary parameters.
We first prove that each root subgroup is preserved and then recover the corresponding maps on root parameters.
The relative root system is of type $B_{\ell-1}$, with long-root parameters belonging to $R_\theta$ and short-root parameters belonging to $R$.
Thus the normalized automorphism initially determines two maps,
\[
    \eta\colon R_\theta\longrightarrow R_\theta,
    \qquad
    \mu\colon R\longrightarrow R.
\]
The mixed commutator relations imply that
\[
    \mu(uv)=\eta(u)\mu(v)
    \qquad (u\in R_\theta,\ v\in R),
\]
which, in particular, gives $\mu|_{R_\theta}=\eta$.
Another mixed commutator relation establishes compatibility with $\theta$ and determines the multiplication of anti-invariant elements.
Together with the additive decomposition
\[
    R=R_\theta\oplus R_\theta^-,
\]
these relations show that $\mu$ is a unital ring homomorphism commuting with $\theta$.
Applying the same argument to the inverse group automorphism shows that $\mu$ is bijective.
Hence the actions on the long-root and short-root subgroups are governed by a single automorphism of the ring with involution $(R,\theta)$.

Finally, we extend the classification from the adjoint elementary subgroup to arbitrary $\rho$-invariant weight lattices and to the corresponding full twisted
Chevalley groups.
The former follows by considering the appropriate central quotients and using the perfectness of the elementary subgroups.
For the full groups, the elementary subgroup is characteristic, so every automorphism restricts to one of the forms already classified, leaving only
a central automorphism.
These arguments follow the approach of \cite[Theorems~3.3 and~3.4]{EB&DM1}.

Thus the main contribution of the present paper lies in establishing the local orthogonal classification.
The key ingredients are the exact normalization of automorphisms, the correction of the conjugating matrix within the orthogonal model, and the recovery of a
single ring automorphism compatible with $\theta$.
The proof treats all ranks $\ell\geq4$, including the lowest-rank case, and allows the induced involution on the residue field to be trivial.

\medskip
\noindent\textbf{Organization.}
The paper is organized as follows.
In Section~\ref{sec:G(R)_and_MT}, we introduce the matrix realizations of the twisted Chevalley groups of type ${}^2D_\ell$ and state our main theorems.
In Section~\ref{sec:outline_of_main_thm}, we reduce the problem to the residue field, construct compatible matrix representatives, and establish the orthogonal correction needed for the proof.
Sections~\ref{sec:image_of_h(-1)}--\ref{sec:image_of_x(1)} are devoted to the successive normalization of the images of $h_{[\alpha]}(-1)$, $w_{[\alpha]}(1)$, and $x_{[\alpha]}(1)$.
Finally, in Section~\ref{sec:image_of_x(t)}, we determine the action on root elements with arbitrary parameters, recover the corresponding ring automorphism, and complete the proof of the main theorem.

Throughout the paper, we follow the general notation and preliminaries established in~\cite[Section~2]{EB&DM1}, introducing the constructions specific to type ${}^2D_\ell$ as needed.


\section{The Group \texorpdfstring{$G_{\pi,\sigma}(D_\ell,R)$}{G(R)} and the Main Theorems}
\label{sec:G(R)_and_MT}

Throughout, $R$ is a commutative local ring with $1/2$,
$J=\Rad(R)$, and $\theta$ is a ring automorphism of order $2$.
We write $\bar r=\theta(r)$ and apply the bar entrywise to matrices.
The automorphism induced by $\theta$ on $R/J$ may be trivial.
We describe the matrix models used below; general preliminaries
and notation are as in \cite[Section~2]{EB&DM1}.


\subsection{The root system of type \texorpdfstring{$D_\ell$}{D-l}}
\label{subsec:root_system}

For $\ell\geq4$, let
\[
    \Phi=\{\pm\varepsilon_i\pm\varepsilon_j
    \mid 1\leq i<j\leq\ell\}\subset\mathbb R^\ell,
\]
where $\varepsilon_1,\ldots,\varepsilon_\ell$ is the standard basis.
We choose the simple system $\Delta=\{\alpha_1,\ldots,\alpha_\ell\}$ with
\[
    \alpha_i=\varepsilon_i-\varepsilon_{i+1}
    \quad(1\leq i\leq\ell-1),
    \qquad
    \alpha_\ell=\varepsilon_{\ell-1}+\varepsilon_\ell.
\]
Fix the diagram involution $\rho$ interchanging
$\alpha_{\ell-1}$ and $\alpha_\ell$ and fixing the remaining
simple roots. Its extension to $\mathbb R^\ell$ fixes
$\varepsilon_1,\ldots,\varepsilon_{\ell-1}$ and sends
$\varepsilon_\ell$ to $-\varepsilon_\ell$.
In particular, for $D_4$ we fix this transposition of the
outer vertices.

The relative twisted root system $\Phi_\rho$ is of type $B_{\ell-1}$:
\[
    \Phi_\rho=
    \{\pm\nu_i\pm\nu_j\mid 1\leq i<j\leq\ell-1\}
    \cup\{\pm\nu_i\mid 1\leq i\leq\ell-1\},
\]
where $\nu_1,\ldots,\nu_{\ell-1}$ is the standard basis of
$\mathbb R^{\ell-1}$. Its simple system
$\Delta_\rho=\{[\alpha_1],\ldots,[\alpha_{\ell-1}]\}$ is given by
\[
    [\alpha_i]=\nu_i-\nu_{i+1}
    \quad(1\leq i\leq\ell-2),
    \qquad
    [\alpha_{\ell-1}]=\nu_{\ell-1}.
\]
Long roots correspond to singleton $\rho$-orbits and are of
type $A_1$; short roots correspond to pairs of orthogonal
roots and are of type $A_1^2$.
Thus $[\alpha_i]=\{\alpha_i\}$ for $i\leq\ell-2$, whereas
$[\alpha_{\ell-1}]=\{\alpha_{\ell-1},\alpha_\ell\}$.

We will also use
\[
    [\widetilde\alpha_i]=\nu_i+\nu_{i+1}
    \quad(1\leq i\leq\ell-2),
    \qquad
    [\beta_i]=\nu_i=\sum_{j=i}^{\ell-1}[\alpha_j]
    \quad(1\leq i\leq\ell-1).
\]
The roots $[\alpha_i]$ and $[\widetilde\alpha_i]$ are strictly
orthogonal, and the $[\beta_i]$ are precisely the positive
short roots.


\subsection{The Lie algebra of type \texorpdfstring{$D_\ell$}{D-l}}
\label{subsec:Lie_algebra}

The complex simple Lie algebra $\mathcal L(D_\ell,\mathbb C)$
is realized as
\[
    \mathfrak{so}_{2\ell}(\mathbb C)
    =\{A\in M_{2\ell}(\mathbb C)\mid A^tS+SA=0\},
    \qquad
    S=\begin{pmatrix}0&I_\ell\\ I_\ell&0\end{pmatrix}.
\]
We use the root vectors
\begin{align*}
    X_{\varepsilon_i-\varepsilon_j}
    &=E_{i,j}-E_{\ell+j,\ell+i}
    &&(1\leq i\ne j\leq\ell),\\
    X_{\varepsilon_i+\varepsilon_j}
    &=E_{i,\ell+j}-E_{j,\ell+i}
    &&(1\leq i<j\leq\ell),\\
    X_{-\varepsilon_i-\varepsilon_j}
    &=E_{\ell+j,i}-E_{\ell+i,j}
    &&(1\leq i<j\leq\ell).
\end{align*}
Together with $H_{\alpha_i}=[X_{\alpha_i},X_{-\alpha_i}]$,
these form a Chevalley basis with the conventions of
\cite[Lemma~2.4]{EB&DM1}. Explicitly,
\begin{align*}
    H_{\alpha_i}
    &=E_{i,i}-E_{i+1,i+1}
      -E_{\ell+i,\ell+i}+E_{\ell+i+1,\ell+i+1}
    \quad(1\leq i\leq\ell-1),\\
    H_{\alpha_\ell}
    &=E_{\ell-1,\ell-1}+E_{\ell,\ell}
      -E_{2\ell-1,2\ell-1}-E_{2\ell,2\ell}.
\end{align*}

The diagram involution is induced by conjugation with
\[
    Q=E_{\ell,2\ell}+E_{2\ell,\ell}
      +\sum_{i\notin\{\ell,2\ell\}}E_{i,i}
    =\begin{pmatrix}
        I_{\ell-1}&0&0&0\\
        0&0&0&1\\
        0&0&I_{\ell-1}&0\\
        0&1&0&0
    \end{pmatrix}.
\]
Here $Q^t=Q=Q^{-1}$, $Q^tSQ=S$, and $SQ=QS$.
The chosen root vectors satisfy
\[
    QX_\alpha Q^{-1}=X_{\rho(\alpha)},
    \qquad
    QH_{\alpha_i}Q^{-1}=H_{\rho(\alpha_i)}.
\]


\subsection{The groups \texorpdfstring{$\SO_{2\ell}(R)$}{SO(R)} and
\texorpdfstring{$\SO_{2\ell}(R)^\sigma$}{SO(R)-sigma}}
\label{subsec:SO(R)}

The split special orthogonal group is
\[
    \SO_{2\ell}(R)
    =\{A\in\SL_{2\ell}(R)\mid A^tSA=S\}.
\]
Let $\pi_0$ be the natural $2\ell$-dimensional representation
of $\mathcal L(D_\ell,\mathbb C)$. Its weight lattice is
\[
    \Lambda_{\pi_0}=\Lambda_r+\mathbb Z\omega_1,
    \qquad \omega_1=\varepsilon_1,
\]
and $G_{\pi_0}(D_\ell,R)\cong\SO_{2\ell}(R)$.
The lattice is $\rho$-invariant, and the corresponding graph
automorphism is $A\mapsto QAQ^{-1}$.
It commutes with the entrywise action of $\theta$, so
\[
    \sigma(A)=Q\overline A Q^{-1}
\]
defines the twisting involution. Thus
\[
    G_{\pi_0,\sigma}(D_\ell,R)
    \cong\SO_{2\ell}(R)^\sigma
    =\{A\in\SO_{2\ell}(R)\mid Q\overline A Q^{-1}=A\}.
\]
Every such matrix also satisfies
\[
    \overline A^{\,t}PA=P,
    \qquad P=SQ=QS.
\]

Set
\[
    R_\theta=\{r\in R\mid\bar r=r\},
    \qquad
    R_\theta^-=\{r\in R\mid\bar r=-r\}.
\]
Since $2$ is invertible, $R=R_\theta\oplus R_\theta^-$ as
$R_\theta$-modules, with projections
$r\mapsto(r+\bar r)/2$ and $r\mapsto(r-\bar r)/2$.
For a relative root $[\alpha]$, put
\[
    R_{[\alpha]}=
    \begin{cases}
        R_\theta,&[\alpha]\sim A_1,\\
        R,&[\alpha]\sim A_1^2,
    \end{cases}
    \qquad
    R_{[\alpha]}^\times=R_{[\alpha]}\cap R^\times.
\]

The root elements in the natural representation are
\[
    x_{[\alpha]}(t)=
    \begin{cases}
        I_{2\ell}+tX_\alpha,&[\alpha]\sim A_1,\\
        (I_{2\ell}+tX_\alpha)
        (I_{2\ell}+\bar tX_{\bar\alpha}),&[\alpha]\sim A_1^2,
    \end{cases}
    \qquad t\in R_{[\alpha]},
\]
where $\bar\alpha=\rho(\alpha)$.
For $[\beta_i]=\nu_i$ we choose
$\alpha=\varepsilon_i-\varepsilon_\ell$, and for $-[\beta_i]$
we use $-\alpha$.
In particular, for long roots and $t\in R_\theta$,
\begin{align*}
    x_{\nu_i-\nu_j}(t)
    &=I_{2\ell}+t(E_{i,j}-E_{\ell+j,\ell+i})
    &&(i\ne j),\\
    x_{\nu_i+\nu_j}(t)
    &=I_{2\ell}+t(E_{i,\ell+j}-E_{j,\ell+i})
    &&(i<j),\\
    x_{-\nu_i-\nu_j}(t)
    &=I_{2\ell}+t(E_{\ell+j,i}-E_{\ell+i,j})
    &&(i<j),
\end{align*}
where all indices lie in $\{1,\ldots,\ell-1\}$.
For short roots, $1\leq i\leq\ell-1$, and $t\in R$,
\begin{align*}
    x_{\nu_i}(t)
    &=I_{2\ell}+t(E_{i,\ell}-E_{2\ell,\ell+i})\\
    &\quad+\bar t(E_{i,2\ell}-E_{\ell,\ell+i})
      -t\bar tE_{i,\ell+i},\\
    x_{-\nu_i}(t)
    &=I_{2\ell}+t(E_{\ell,i}-E_{\ell+i,2\ell})\\
    &\quad+\bar t(E_{2\ell,i}-E_{\ell+i,\ell})
      -t\bar tE_{\ell+i,i}.
\end{align*}
These elements generate $E'_{\pi_0,\sigma}(D_\ell,R)$.
For $t\in R_{[\alpha]}^\times$, set
\[
    w_{[\alpha]}(t)
    =x_{[\alpha]}(t)x_{-[\alpha]}(-t^{-1})x_{[\alpha]}(t),
    \qquad
    h_{[\alpha]}(t)=w_{[\alpha]}(t)w_{[\alpha]}(-1).
\]

Let $e_1,\ldots,e_{2\ell}$ be the standard basis and put
\[
    B_i=Re_i\oplus Re_{\ell+i},
    \qquad
    R^{2\ell}=B_1\oplus\cdots\oplus B_\ell.
\]
In block matrices relative to this decomposition, the ordered
basis of each $B_i$ is $(e_i,e_{\ell+i})$.
All matrices below act identically outside the indicated summands.
For $1\leq i\leq\ell-2$ and $t\in R_\theta$,
\[
    x_{[\alpha_i]}(t)|_{B_i\oplus B_{i+1}}
    =\begin{pmatrix}
        1&0&t&0\\
        0&1&0&0\\
        0&0&1&0\\
        0&-t&0&1
    \end{pmatrix},
\]
\[
    w_{[\alpha_i]}(1)|_{B_i\oplus B_{i+1}}
    =\begin{pmatrix}0&I_2\\-I_2&0\end{pmatrix},
    \qquad
    h_{[\alpha_i]}(-1)|_{B_i\oplus B_{i+1}}=-I_4.
\]
For the short simple root and $t\in R$,
\[
    x_{[\alpha_{\ell-1}]}(t)|_{B_{\ell-1}\oplus B_\ell}
    =\begin{pmatrix}
        1&-t\bar t&t&\bar t\\
        0&1&0&0\\
        0&-\bar t&1&0\\
        0&-t&0&1
    \end{pmatrix},
\]
\[
    w_{[\alpha_{\ell-1}]}(1)|_{B_{\ell-1}\oplus B_\ell}
    =\begin{pmatrix}
        0&-1&0&0\\
        -1&0&0&0\\
        0&0&0&-1\\
        0&0&-1&0
    \end{pmatrix}.
\]
In this case $w_{[\alpha_{\ell-1}]}(1)^2=I_{2\ell}$ and
$h_{[\alpha_{\ell-1}]}(-1)=I_{2\ell}$.


\subsection{The groups \texorpdfstring{$\PSO_{2\ell}(R)$}{PSO(R)} and
\texorpdfstring{$\PSO_{2\ell}(R)^\sigma$}{PSO(R)-sigma}}
\label{subsec:PSO(R)}

We write $\PSO_{2\ell}(R)$ for the $R$-points of the split
adjoint group scheme
$\PSO_{2\ell}=\SO_{2\ell}/\mu_2$ over $\mathbb Z[1/2]$.
Similarly, $\PGL_{2\ell}(R)$ denotes the $R$-points of the
projective general linear group scheme.
Since $R$ is local, $\operatorname{Pic}(R)=0$, and hence
\[
    \PGL_{2\ell}(R)
    =\GL_{2\ell}(R)/(R^\times I_{2\ell}).
\]
We denote the class of a matrix $A$ by $[A]$.

Consider the orthogonal similarities and proper orthogonal
similarities
\begin{align*}
    \mathrm{GO}_{2\ell}(R)
    &=\{A\in\GL_{2\ell}(R)\mid
        A^tSA=\lambda S\text{ for some }\lambda\in R^\times\},\\
    \mathrm{GO}_{2\ell}^+(R)
    &=\{A\in\mathrm{GO}_{2\ell}(R)\mid
        \det A=\lambda^\ell,\text{ where }A^tSA=\lambda S\}.
\end{align*}
The multiplier $\lambda$ is uniquely determined by $A$.
The group scheme identity
$\PSO_{2\ell}\cong\mathrm{GO}_{2\ell}^+/\mathbb G_m$
and the vanishing of $\operatorname{Pic}(R)$ give
\[
    \PSO_{2\ell}(R)
    \cong\mathrm{GO}_{2\ell}^+(R)/(R^\times I_{2\ell})
    \subseteq\PGL_{2\ell}(R).
\]
Thus the adjoint group is represented by projective classes
of proper orthogonal similarities.

This description also makes the lifting obstruction explicit.
If $A^tSA=\lambda S$ and $\det A=\lambda^\ell$, then
$[A]$ has a representative in $\SO_{2\ell}(R)$ if and only if
$\lambda\in(R^\times)^2$: a scalar multiple $aA$ is orthogonal
exactly when $a^2\lambda=1$, and then its determinant is $1$.
Consequently, the natural map
\[
    \SO_{2\ell}(R)/\{\pm I_{2\ell}\}
    \longrightarrow\PSO_{2\ell}(R)
\]
need not be surjective, even over a field.

The involution $\sigma$ descends to the projective group, and
\[
    G_{\mathrm{ad}}(D_\ell,R)\cong\PSO_{2\ell}(R),
    \qquad
    G_{\mathrm{ad},\sigma}(D_\ell,R)
    \cong\PSO_{2\ell}(R)^\sigma.
\]
In the preceding projective realization,
\[
    \PSO_{2\ell}(R)^\sigma
    =\left\{[A]\in\PGL_{2\ell}(R)\ \middle|\
    \begin{gathered}
        A^tSA=\lambda S,\qquad \det A=\lambda^\ell,\\
        Q\overline A Q^{-1}=\mu A
        \quad\text{for some }\lambda,\mu\in R^\times
    \end{gathered}\right\}.
\]
Here the fixed-point condition is imposed on the projective
class, not on the chosen matrix.

For elementary groups, \cite[Lemma~2.7]{EB&DM1} gives
\[
    E'_{\mathrm{ad},\sigma}(D_\ell,R)
    \cong E'_{\pi_0,\sigma}(D_\ell,R)
    /Z\bigl(E'_{\pi_0,\sigma}(D_\ell,R)\bigr).
\]
In particular, every adjoint elementary element has a
representative in $E'_{\pi_0,\sigma}(D_\ell,R)$.
These representatives belong to $\SO_{2\ell}(R)^\sigma$.
The kernel of this projection consists of the scalar matrices
in $E'_{\pi_0,\sigma}(D_\ell,R)$ and is contained in
$\{\pm I_{2\ell}\}$.
We use the same notation $x_{[\alpha]}(t)$,
$w_{[\alpha]}(t)$, and $h_{[\alpha]}(t)$ for the standard
matrices and their projective classes, writing brackets
when the distinction is needed.


\subsection{\texorpdfstring{$\sigma$}{sigma}-symmetric matrices}
\label{subsec:sigma_sym_matrices}

The map $\sigma(A)=Q\overline A Q^{-1}$ is also an involution
of the matrix ring. Put
\[
    M_{2\ell}(R)^\sigma
    =\{A\in M_{2\ell}(R)\mid\sigma(A)=A\},
    \qquad
    \GL_{2\ell}(R)^\sigma
    =\GL_{2\ell}(R)\cap M_{2\ell}(R)^\sigma.
\]
We call these matrices \emph{$\sigma$-symmetric}.
The following averaging argument applies simultaneously to
any finite family of exact intertwining identities.

\begin{lemma}\label{lemma:Q-symmetrization}
    Let $A_i,B_i\in M_{2\ell}(R)^\sigma$ for $1\leq i\leq m$.
    Suppose that $C\in\GL_{2\ell}(R)$ satisfies
    \[
        C\equiv I_{2\ell}\pmod J,
        \qquad A_iC=CB_i\quad(1\leq i\leq m).
    \]
    Then
    \[
        \widetilde C=\frac{C+\sigma(C)}2
    \]
    belongs to $\GL_{2\ell}(R)^\sigma$, is congruent to
    $I_{2\ell}$ modulo $J$, and satisfies
    $A_i\widetilde C=\widetilde C B_i$ for every $i$.
\end{lemma}

\begin{proof}
    Since $\theta(J)=J$, we have
    $\widetilde C\equiv I_{2\ell}\pmod J$, so
    $\widetilde C$ is invertible; also
    $\sigma(\widetilde C)=\widetilde C$.
    Applying $\sigma$ to $A_iC=CB_i$ gives
    $A_i\sigma(C)=\sigma(C)B_i$.
    Adding the two identities and dividing by $2$ proves
    $A_i\widetilde C=\widetilde C B_i$ for every $i$.
\end{proof}

In applying the lemma, both the source and target matrices
must be $\sigma$-symmetric, and the intertwining identities
must hold in matrices, not merely in projective classes.

\begin{rmk}
    Let $F\in M_{2\ell}(R)$ be an invertible symmetric matrix.
    Suppose that $A,B\in\GL_{2\ell}(R)$ satisfy
    \[
        A^tFA=B^tFB=F,
        \qquad [A]=[B],
        \qquad A\equiv B\pmod J.
    \]
    Then $A=B$. Indeed, writing $A=\lambda B$ gives
    $\lambda^2=1$ and $\lambda\equiv1\pmod J$.
    Since $\lambda+1\equiv2\pmod J$ is a unit,
    $(\lambda-1)(\lambda+1)=0$ implies $\lambda=1$.

    Thus a projective identity between words in matrices
    preserving the same form is exact whenever the two words
    have the same reduction. This observation remains valid
    after a simultaneous change of basis $C$: the conjugate
    matrices $C^{-1}AC$ preserve the form $C^tFC$.
\end{rmk}


\subsection{The main theorems}
\label{subsec:MT}

We use the terminology for standard automorphisms from
\cite[Section~2]{EB&DM1}.
In particular, ring automorphisms are induced by automorphisms
of $R$ commuting with $\theta$.
The defining diagram involution contributes no separate
factor: on the $\sigma$-fixed group, $\rho(g)=\theta(g)$.

\begin{thm}\label{MT_local1}
    Let $R$ be a commutative local ring with $1/2$, and let
    $G=E'_{\mathrm{ad},\sigma}(D_\ell,R)$, where $\ell\geq4$.
    Every automorphism $\varphi$ of $G$ has the form
    \[
        \varphi=i_g\circ\mu,
        \qquad g\in G_{\mathrm{ad},\sigma}(D_\ell,R),
    \]
    where $\mu$ is induced by an automorphism of $R$ commuting
    with $\theta$. Thus every automorphism of $G$ is standard,
    with a strictly inner and a ring factor.
\end{thm}

\begin{thm}\label{MT_local2}
    Let $R$ be a commutative local ring with $1/2$, let
    $\ell\geq4$, and let $\pi$ be a representation whose
    weight lattice $\Lambda_\pi$ is $\rho$-invariant.
    Every automorphism of $E'_{\pi,\sigma}(D_\ell,R)$ is
    standard: it is a composition of a strictly inner,
    a diagonal, and a ring automorphism.
\end{thm}

\begin{thm}\label{MT_local3}
    Let $R$ be a commutative local ring with $1/2$, let
    $\ell\geq4$, and let $\pi$ be a representation whose
    weight lattice $\Lambda_\pi$ is $\rho$-invariant.
    Every automorphism of $G_{\pi,\sigma}(D_\ell,R)$ is
    standard: it is a composition of a strictly inner,
    a diagonal, a central, and a ring automorphism.
\end{thm}

Theorem~\ref{MT_local1} is proved in
Sections~\ref{sec:outline_of_main_thm}--\ref{sec:image_of_x(t)}.
The reduction follows the strategy of
\cite[Theorem~3.2]{EB&DM1}; the orthogonal correction and the
matrix calculations specific to $D_\ell$ are given below.
Theorems~\ref{MT_local2} and~\ref{MT_local3} then follow by the
arguments of \cite[Theorems~3.3 and~3.4]{EB&DM1}, respectively.


\section{Reduction of the Main Theorem}
\label{sec:outline_of_main_thm}

We retain the assumptions and notation of
Section~\ref{sec:G(R)_and_MT}. Put $k=R/J$, let
$q\colon R\to k$ be the quotient map, and denote the induced
involution of $k$ by $\theta_k$. We also use $q$ for reduction
of matrices and group elements, and $\sigma$ for the induced
twisting over $k$.
Throughout this section, for $A=R,k$, write
\[
    E(A)=E'_{\mathrm{ad},\sigma}(D_\ell,A),
    \qquad
    G(A)=G_{\mathrm{ad},\sigma}(D_\ell,A).
\]


\subsection{Reduction modulo the radical}

\begin{lemma}[{\cite[Lemma~4.1]{EB&DM1}}]
\label{lemma:maximal_normal_subgroup}
    Let
    $N_J=E(R)\cap G_{\mathrm{ad},\sigma}(D_\ell,J)$.
    Then $N_J$ is the greatest proper normal subgroup of
    $E(R)$, and reduction induces an isomorphism
    \[
        E(R)/N_J\cong E(k).
    \]
\end{lemma}

Here reduction is surjective on every relative root subgroup.
For a long-root parameter, lift an element of $k_{\theta_k}$
to $r\in R$ and replace $r$ by $(r+\bar r)/2$; short-root
parameters lift directly. This applies also when
$\theta_k=\mathrm{id}_k$.

Let $\varphi\in\operatorname{Aut}(E(R))$.
Since $N_J$ is characteristic, $\varphi$ induces an automorphism
$\bar\varphi$ of $E(k)$.
There are two cases in applying the field classification.

If $\theta_k\ne\mathrm{id}_k$, then $E(k)$ is an elementary
adjoint twisted group of type ${}^2D_\ell$.
By \cite[Theorem~2.10]{EB&DM1},
\[
    \bar\varphi=i_g\circ f,
    \qquad g\in G(k),
    \qquad f\in\operatorname{Aut}(k),
    \qquad f\theta_k=\theta_k f.
\]
The inner and diagonal factors are combined in $g$.
The defining graph involution agrees with $\theta_k$ on the
fixed-point group and is absorbed into $f$.

If $\theta_k=\mathrm{id}_k$, use the orthogonal decomposition
\[
    k^{2\ell}=V_+\oplus V_-,
    \qquad V_\pm=\ker(Q\mp I),
    \qquad V_-=k(e_\ell-e_{2\ell}).
\]
The restriction $S_+=S|_{V_+}$ has dimension $2\ell-1$
and Witt index $\ell-1$.
The root matrices of Subsection~\ref{subsec:SO(R)} fix $V_-$
pointwise, and their restrictions to $V_+$ are the root
subgroups of the split group of type $B_{\ell-1}$.
Projectivization is injective on this subgroup, since its
matrices fix $V_-$ pointwise. Thus
\[
    E(k)\cong E_{\mathrm{ad}}(B_{\ell-1},k).
\]
This identification is defined over the prime field.
The untwisted field classification~\cite{RS} gives an
inner-diagonal factor induced by an element
$g_+\in\SO(V_+,S_+)$ and a field automorphism $f$ of $k$.
There is no graph factor for $B_{\ell-1}$,
since $\ell-1\geq3$ and $\operatorname{char}k\ne2$.
Extending $g_+$ by the identity on $V_-$ gives
\[
    g=[g_+\oplus1]\in G(k).
\]
Hence the same decomposition $\bar\varphi=i_g\circ f$ holds,
and $f\theta_k=\theta_k f$ is automatic.

We next construct a lift $g_1\in G(R)$ of $g$ under the reduction map $q:G(R)\to G(k)$. When $\theta_k\ne\mathrm{id}_k$, its existence follows from \cite[Lemma~2.6]{EB&DM1}. Suppose now that $\theta_k=\mathrm{id}_k$. Then the fixed subring $R_\theta$ is local with residue field $k$. The split special orthogonal group on
\[
    \bigoplus_{i=1}^{\ell-1}
    (R_\theta e_i\oplus R_\theta e_{\ell+i})
    \oplus R_\theta(e_\ell+e_{2\ell})
\]
embeds in $\SO_{2\ell}(R)^\sigma$ by extension of scalars and by acting trivially on $R(e_\ell-e_{2\ell})$. 
A Bruhat decomposition of $g_+$ lifts to this group, since its root parameters lift to $R_\theta$ and its torus parameters lift to $R_\theta^\times$. 
The projective class of the resulting matrix gives $g_1\in G(R)$ satisfying $q(g_1)=g$.

Set
\[
    \varphi_1=i_{g_1}^{-1}\circ\varphi.
\]
Then $\varphi_1\in\operatorname{Aut}(E(R))$ induces the
entrywise field automorphism $f$ modulo $J$.
No lift of $f$ to an automorphism of $R$ is assumed.


\subsection{Compatible matrix representatives}

Put $E_0(A)=E'_{\pi_0,\sigma}(D_\ell,A)$ for $A=R,k$,
and let $\delta_A\colon E_0(A)\to E(A)$ be the natural
projection. These maps are surjective and commute with
reduction. Their kernels consist of scalar matrices and are
contained in $\{\pm I_{2\ell}\}$.

\begin{lemma}\label{lemma:compatible-representatives}
    For every $x\in E_0(R)$, there is a unique $y\in E_0(R)$
    satisfying
    \[
        \delta_R(y)=\varphi_1\bigl(\delta_R(x)\bigr),
        \qquad q(y)=f\bigl(q(x)\bigr).
    \]
    The assignment $x\mapsto y$ defines an automorphism
    $\widehat\varphi_1$ of $E_0(R)$ lifting $\varphi_1$.
\end{lemma}

\begin{proof}
    We first recall the commutator argument of
    \cite[Lemma~4.2]{EB&DM1}.
    The group $E_0(R)$ is perfect. Indeed, since $\ell-1\geq3$,
    for every relative root $[\alpha]$, there is a long root
    $[\beta]$ such that $h=h_{[\beta]}(-1)$ inverts the root
    subgroup $x_{[\alpha]}(R_{[\alpha]})$. Consequently,
    \[
        x_{[\alpha]}(t)=[h,x_{[\alpha]}(-t/2)].
    \]
    Write $x=\prod_s[a_s,b_s]$ and choose lifts
    $A_s,B_s\in E_0(R)$ of
    $\varphi_1(\delta_R(a_s))$ and
    $\varphi_1(\delta_R(b_s))$, respectively.
    Their reductions differ from $f(q(a_s))$ and
    $f(q(b_s))$ by elements of $\ker\delta_k$.
    These elements are central, so
    \[
        y=\prod_s[A_s,B_s]
    \]
    has the required projective class and satisfies
    \[
        q(y)
        =\prod_s[f(q(a_s)),f(q(b_s))]
        =f(q(x)).
    \]

    If $y$ and $z$ are two such lifts, then $yz^{-1}$ is a
    scalar orthogonal matrix congruent to the identity
    modulo $J$. Thus $yz^{-1}=\lambda I$, where
    $\lambda^2=1$ and $\lambda\equiv1\pmod J$.
    Since $\lambda+1$ is a unit, $\lambda=1$.
    This proves uniqueness.

    The product of the lifts of $x_1$ and $x_2$ satisfies the
    defining conditions for the lift of $x_1x_2$.
    Uniqueness therefore makes $\widehat\varphi_1$ a
    homomorphism. Applying the same construction to
    $\varphi_1^{-1}$ and $f^{-1}$ gives its inverse.
\end{proof}

\begin{rmk}
\normalfont
Let
\[
    \GL_{2\ell}(R,J)
    =\{C\in\GL_{2\ell}(R)\mid C\equiv I_{2\ell}\pmod J\}.
\]
For $C\in\GL_{2\ell}(R,J)$ and
$\psi=i_{[C]}^{-1}\circ\varphi_1$, we use the representatives
\[
    Y_C(x)=C^{-1}\widehat\varphi_1(x)C
    \qquad(x\in E_0(R)).
\]
They satisfy
\[
    [Y_C(x)]=\psi([x]),
    \qquad q(Y_C(x))=f(q(x)),
    \qquad Y_C(x)^t(C^tSC)Y_C(x)=C^tSC.
\]
The map $Y_C$ is a homomorphism, so all exact matrix
relations in $E_0(R)$ remain exact for these representatives.
In particular, no separate scalar rescaling is needed for
the involution, Weyl, or root relations used below.
The map $\psi$ need not be an automorphism of $E(R)$:
its image may be a conjugate subgroup of $\PGL_{2\ell}(R)$,
and the preserved form is $C^tSC$, not necessarily $S$.

We use this convention throughout
Sections~\ref{sec:image_of_h(-1)}--\ref{sec:image_of_x(t)}.
Once an image has been normalized as a matrix, that same
matrix is retained in all subsequent relations.
\end{rmk}


\subsection{Normalization and orthogonal correction}

The first three steps are changes of basis congruent to the
identity modulo $J$.

\medskip

\noindent\textbf{Step 1.}
Find $g_2\in\GL_{2\ell}(R,J)$ such that
$\varphi_2=i_{[g_2]}^{-1}\circ\varphi_1$ satisfies
\[
    \varphi_2(h_{[\alpha]}(-1))=h_{[\alpha]}(-1)
    \qquad([\alpha]\in\Phi_\rho).
\]

\noindent\textbf{Step 2.}
Find $g_3\in\GL_{2\ell}(R,J)$ preserving this normalization
such that $\varphi_3=i_{[g_3]}^{-1}\circ\varphi_2$ satisfies
\[
    \varphi_3(w_{[\alpha]}(1))=w_{[\alpha]}(1)
    \qquad([\alpha]\in\Phi_\rho).
\]

\noindent\textbf{Step 3.}
Find $g_4\in\GL_{2\ell}(R,J)$ preserving the fixed Weyl
elements such that $\varphi_4=i_{[g_4]}^{-1}\circ\varphi_3$
satisfies
\[
    \varphi_4(x_{[\alpha]}(1))=x_{[\alpha]}(1)
    \qquad([\alpha]\in\Phi_\rho).
\]

\medskip

These steps are proved in
Sections~\ref{sec:image_of_h(-1)},
\ref{sec:image_of_w(1)}, and~\ref{sec:image_of_x(1)},
respectively. We now explain how their product is corrected
to an element of $G(R)$.

\begin{lemma}\label{lemma:centralizer_of_x(1)}
    Let $M\in\GL_{2\ell}(R,J)^\sigma$, and suppose that
    \[
        x_{[\alpha]}(1)M
        =\lambda_{[\alpha]}Mx_{[\alpha]}(1)
        \qquad([\alpha]\in\Phi_\rho)
    \]
    for units $\lambda_{[\alpha]}\in R^\times$.
    Then every $\lambda_{[\alpha]}=1$, and
    \[
        M=aI_{2\ell}+bQ,
        \qquad a,b\in R_\theta,
        \qquad a\equiv1\pmod J,
        \qquad b\in J.
    \]
    Conversely, every matrix $aI_{2\ell}+bQ$ commutes with
    all $x_{[\alpha]}(1)$; it is $\sigma$-symmetric if and
    only if $a,b\in R_\theta$.
\end{lemma}

\begin{proof}
    Each root matrix is the identity outside two of the
    summands $B_i$. Choose a coordinate $j$ outside its
    support. The $(j,j)$-entry of the assumed identity is
    \[
        M_{jj}=\lambda_{[\alpha]}M_{jj}.
    \]
    Since $M_{jj}\equiv1\pmod J$, it is a unit, and hence
    $\lambda_{[\alpha]}=1$.

    Thus $M$ commutes with all standard Weyl elements and
    with $h_{[\alpha_i]}(-1)$ for $1\leq i\leq\ell-2$.
    The simultaneous sign characters of these involutions
    on $B_1,\ldots,B_\ell$ are pairwise distinct.
    Since $2$ is invertible, $M$ is block diagonal.
    Commuting with the long simple Weyl elements identifies
    its first $\ell-1$ blocks, while the short simple Weyl
    element makes the last two blocks commute with
    $J_2=\begin{pmatrix}0&1\\1&0\end{pmatrix}$.
    Therefore
    \[
        M=\operatorname{diag}(M',\ldots,M',M''),
        \qquad
        M'=pI_2+qJ_2,
        \qquad M''=rI_2+sJ_2.
    \]
    Commuting with $x_{[\alpha_1]}(1)$ forces $q=0$.
    The $(1,3)$-entry of the commutation relation on
    $B_{\ell-1}\oplus B_\ell$ with
    $x_{[\alpha_{\ell-1}]}(1)$ then gives $p=r+s$.
    Hence $M=rI_{2\ell}+sQ$.
    The conditions $\sigma(M)=M$ and $M\equiv I\pmod J$
    give $r,s\in R_\theta$, $r\equiv1\pmod J$, and $s\in J$.

    Finally, $Q$ commutes with every $x_{[\alpha]}(1)$,
    because these matrices have integer entries and are
    $\sigma$-fixed. This proves the converse; comparison
    of the entries of $aI+bQ$ gives the last assertion.
\end{proof}

\begin{prop}\label{prop:normalization_of_C}
    Let $C\in\GL_{2\ell}(R,J)^\sigma$, and suppose that
    \[
        [C]x_{[\alpha]}(1)[C]^{-1}\in G(R)
        \qquad([\alpha]\in\Phi_\rho).
    \]
    There exists $d\in J\cap R_\theta$ such that
    \[
        C'=C(I_{2\ell}+dQ)\in\GL_{2\ell}(R,J)^\sigma,
        \qquad [C']\in G(R).
    \]
    The matrices $C$ and $C'$ induce the same conjugation
    on every $x_{[\alpha]}(1)$.
\end{prop}

\begin{proof}
    Measure the failure of $C$ to preserve the form by
    \[
        M(C)=S^{-1}C^tSC.
    \]
    Since $\sigma(S)=S$ and $\sigma(C)=C$, this matrix belongs
    to $\GL_{2\ell}(R,J)^\sigma$.
    For $x=x_{[\alpha]}(1)$, the hypothesis implies that
    $Y=CxC^{-1}$ satisfies $Y^tSY=\lambda_{[\alpha]}S$.
    Using $x^tSx=S$, we obtain
    \[
        x^{-1}M(C)x=\lambda_{[\alpha]}M(C).
    \]
    Lemma~\ref{lemma:centralizer_of_x(1)} gives
    \[
        M(C)=aI_{2\ell}+bQ,
        \qquad a,b\in R_\theta,
        \qquad a\equiv1\pmod J,
        \qquad b\in J.
    \]

    Taking determinants yields
    \[
        (\det C)^2=(a+b)^{2(\ell-1)}(a^2-b^2).
    \]
    Since $\det C\in R_\theta$ and $a+b\equiv1\pmod J$,
    the element
    \[
        s=\frac{\det C}{(a+b)^{\ell-1}}
        \in R_\theta^\times
    \]
    satisfies $s^2=a^2-b^2$ and $s\equiv1\pmod J$.
    Thus the required square root is already determined
    by $C$.
    Set
    \[
        d=-\frac{b}{a+s}\in J\cap R_\theta,
        \qquad C'=C(I_{2\ell}+dQ).
    \]
    The denominator is a unit because $a+s\equiv2\pmod J$.
    Since $Q^tS=SQ$ and $Q^2=I$, we have
    \begin{align*}
        M(C')
        &=(I+dQ)M(C)(I+dQ)\\
        &=\bigl((1+d^2)a+2db\bigr)I
          +\bigl((1+d^2)b+2da\bigr)Q.
    \end{align*}
    The coefficient of $Q$ vanishes:
    \[
        (1+d^2)b+2da
        =\frac{b(s^2-a^2+b^2)}{(a+s)^2}=0.
    \]
    Hence $M(C')=\lambda I$, where
    $\lambda=(1+d^2)a+2db\equiv1\pmod J$, and therefore
    \[
        (C')^tSC'=\lambda S.
    \]
    The unit $r=\det C'/\lambda^\ell$ satisfies
    $r^2=1$ and $r\equiv1\pmod J$, so $r=1$.
    Thus $C'$ is a proper orthogonal similarity.
    Together with $\sigma(C')=C'$, this gives $[C']\in G(R)$.
    The final assertion follows because $I+dQ$ commutes
    with every $x_{[\alpha]}(1)$.
\end{proof}

Assume now that Steps~1--3 have been carried out, and put
\[
    D=g_2g_3g_4\in\GL_{2\ell}(R,J).
\]
We first make the final normalization exact in matrices.
For each $[\alpha]$, Step~3 gives
\[
    Y_D(x_{[\alpha]}(1))=\lambda x_{[\alpha]}(1),
    \qquad \lambda\equiv1\pmod J.
\]
Choose a long root $[\beta]$ such that
$h=h_{[\beta]}(-1)$ inverts $x=x_{[\alpha]}(1)$.
Since $\varphi_4$ fixes all root elements with parameter $1$,
it also fixes $h$ projectively, so $Y_D(h)=\varepsilon h$
for some unit $\varepsilon$.
Applying the homomorphism $Y_D$ to $hxh^{-1}=x^{-1}$ gives
\[
    \lambda x^{-1}=\lambda^{-1}x^{-1}.
\]
Thus $\lambda^2=1$, and $\lambda\equiv1\pmod J$ implies
$\lambda=1$. Consequently,
\[
    \widehat\varphi_1(x_{[\alpha]}(1))D
    =Dx_{[\alpha]}(1)
    \qquad([\alpha]\in\Phi_\rho).
\]
Both families in these exact identities are
$\sigma$-symmetric. Lemma~\ref{lemma:Q-symmetrization}
therefore applies simultaneously to all of them and gives
\[
    C_1=\frac{D+\sigma(D)}2
    \in\GL_{2\ell}(R,J)^\sigma,
    \qquad
    \widehat\varphi_1(x_{[\alpha]}(1))C_1
    =C_1x_{[\alpha]}(1).
\]
In particular,
$[C_1]x_{[\alpha]}(1)[C_1]^{-1}
=\varphi_1(x_{[\alpha]}(1))\in E(R)$.
By Proposition~\ref{prop:normalization_of_C}, there is
$d\in J\cap R_\theta$ such that
\[
    C_2=C_1(I_{2\ell}+dQ),
    \qquad [C_2]\in G(R),
\]
with the same conjugation action on all $x_{[\alpha]}(1)$.
It follows that
\[
    \varphi_5=i_{[C_2]}^{-1}\circ\varphi_1
    \in\operatorname{Aut}(E(R))
\]
fixes every $x_{[\alpha]}(1)$, and hence every
$w_{[\alpha]}(1)$ and $h_{[\alpha]}(-1)$.
It still induces $f$ modulo $J$.

\medskip

\noindent\textbf{Step 4.}
Show that there exists $\mu\in\operatorname{Aut}(R)$ with
$\mu\theta=\theta\mu$ such that
\[
    \varphi_5(x_{[\alpha]}(t))
    =x_{[\alpha]}(\mu(t))
    \qquad([\alpha]\in\Phi_\rho,\ t\in R_{[\alpha]}).
\]

\medskip

Section~\ref{sec:image_of_x(t)} proves this last step.
Since the relative root subgroups generate $E(R)$, it gives
$\varphi_5=\mu$ on $E(R)$, and therefore
\[
    \varphi=i_{g_1}\circ i_{[C_2]}\circ\mu
           =i_{g_1[C_2]}\circ\mu,
    \qquad g_1[C_2]\in G(R).
\]
Thus Steps~1--4 prove Theorem~\ref{MT_local1}.


\section{The Images of \texorpdfstring{$h_{[\alpha]}(-1)$}{h(-1)}}
\label{sec:image_of_h(-1)}

We use the matrix lift $\widehat\varphi_1$ from
Lemma~\ref{lemma:compatible-representatives}.
For $1\leq i\leq\ell-2$, set
\[
    h_i=\widehat\varphi_1\bigl(h_{[\alpha_i]}(-1)\bigr).
\]
The standard matrices $h_{[\alpha_i]}(-1)$ are commuting
involutions with diagonal entries in $\{1,-1\}$.
Since $\widehat\varphi_1$ is a homomorphism and its reduction
is the entrywise field automorphism $f$, we have
\[
    h_i^2=I_{2\ell},
    \qquad h_i h_j=h_j h_i,
    \qquad h_i\equiv h_{[\alpha_i]}(-1)\pmod J.
\]
Thus no scalar adjustment of these representatives is needed.

Write
\[
    h_{[\alpha_i]}(-1)e_j=\varepsilon_{ij}e_j,
    \qquad \varepsilon_{ij}\in\{1,-1\},
\]
and define
\[
    e_j'=\left(
        \prod_{i=1}^{\ell-2}
        \frac{I_{2\ell}+\varepsilon_{ij}h_i}{2}
    \right)e_j
    \qquad(1\leq j\leq2\ell).
\]
The factors commute, and
$(I_{2\ell}+\varepsilon_{ij}h_i)/2$ projects onto the
$\varepsilon_{ij}$-eigensubmodule of $h_i$. Hence
\[
    h_i e_j'=\varepsilon_{ij}e_j'
    \qquad(1\leq i\leq\ell-2,\ 1\leq j\leq2\ell).
\]
Moreover, each factor fixes the image of $e_j$ modulo $J$,
so $e_j'\equiv e_j\pmod J$.
Therefore the matrix $g_2$ with columns $e_1',\ldots,e_{2\ell}'$
is congruent to $I_{2\ell}$ modulo $J$ and is invertible.
Thus $g_2\in\GL_{2\ell}(R,J)$ and
\[
    g_2^{-1}h_i g_2=h_{[\alpha_i]}(-1)
    \qquad(1\leq i\leq\ell-2).
\]

This also normalizes the elements corresponding to all
remaining relative roots. Indeed, the matrix formulas of
Subsection~\ref{subsec:SO(R)} give
\[
    h_{\pm\nu_a\pm\nu_b}(-1)
    =\prod_{i=a}^{b-1}h_{[\alpha_i]}(-1)
    \qquad(1\leq a<b\leq\ell-1):
\]
both sides act by $-I_2$ on $B_a$ and $B_b$, and identically
on the other summands. For every short root,
\[
    h_{\pm\nu_a}(-1)=I_{2\ell}
    \qquad(1\leq a\leq\ell-1).
\]
Since $\widehat\varphi_1$ is a homomorphism, it follows that
\[
    g_2^{-1}\widehat\varphi_1\bigl(h_{[\alpha]}(-1)\bigr)g_2
    =h_{[\alpha]}(-1)
    \qquad([\alpha]\in\Phi_\rho).
\]

Set
\[
    \varphi_2=i_{[g_2]}^{-1}\circ\varphi_1.
\]
Then $\varphi_2$ fixes every $h_{[\alpha]}(-1)$ and still
induces $f$ modulo $J$.
As explained in Section~\ref{sec:outline_of_main_thm},
$\varphi_2$ is regarded as an isomorphism onto a conjugate
subgroup of $\PGL_{2\ell}(R)$.
This proves Step~1 of Section~\ref{sec:outline_of_main_thm}.


\section{The Images of \texorpdfstring{$w_{[\alpha]}(1)$}{w(1)}}
\label{sec:image_of_w(1)}

For each $[\alpha]\in\Phi_\rho$, use the representative
\[
    w_{[\alpha]}
    =g_2^{-1}\widehat\varphi_1\bigl(w_{[\alpha]}(1)\bigr)g_2
    \in\SL_{2\ell}(R)
\]
of $\varphi_2(w_{[\alpha]}(1))$, as in
Lemma~\ref{lemma:compatible-representatives}.
By Section~\ref{sec:image_of_h(-1)},
\[
    w_{[\alpha]}\equiv w_{[\alpha]}(1)\pmod J,
    \qquad w_{[\alpha]}^2=h_{[\alpha]}(-1).
\]
All relations below are exact matrix identities, transported
by the homomorphism $x\mapsto g_2^{-1}\widehat\varphi_1(x)g_2$.
We use the ordered decomposition
$R^{2\ell}=B_1\oplus\cdots\oplus B_\ell$ from
Subsection~\ref{subsec:SO(R)} and write
\[
    J_2=\begin{pmatrix}0&1\\1&0\end{pmatrix}.
\]
We will repeatedly use the following observation: if
$U^2=I$ and $U\equiv I\pmod J$, then $U=I$, since
$U+I$ is invertible and $(U-I)(U+I)=0$.


\subsection{Normalization for short roots}

For $[\beta_i]=\nu_i$, the standard matrix
$w_{[\beta_i]}(1)$ acts by $-J_2$ on $B_i$ and $B_\ell$,
and identically on the remaining summands.
These matrices commute with all $h_{[\alpha_j]}(-1)$ and
with one another. Hence
\[
    w_{[\beta_i]}h_{[\alpha_j]}(-1)
    =h_{[\alpha_j]}(-1)w_{[\beta_i]},
    \qquad
    w_{[\beta_i]}w_{[\beta_k]}
    =w_{[\beta_k]}w_{[\beta_i]}.
\]
The involutions $h_{[\alpha_j]}(-1)$, $1\leq j\leq\ell-2$,
have distinct joint sign patterns on $B_1,\ldots,B_\ell$.
Comparing off-diagonal blocks and using $2\in R^\times$
therefore shows that every $w_{[\beta_i]}$ preserves each
$B_r$. On $B_r$ with $r\notin\{i,\ell\}$, its restriction
is an involution congruent to the identity and hence equals
$I_2$.

Put
\[
    U_i=w_{[\beta_i]}|_{B_i},
    \qquad V_i=w_{[\beta_i]}|_{B_\ell}
    \qquad(1\leq i\leq\ell-1).
\]
Then
\[
    U_i^2=V_i^2=I_2,
    \qquad U_i\equiv V_i\equiv-J_2\pmod J.
\]
The matrices $V_i$ commute, so $V_iV_k$ is an involution
congruent to $I_2$ modulo $J$. Thus $V_iV_k=I_2$ and
$V_i=V_k$. Denote their common value by $U_\ell$.

Define
\[
    T'=\operatorname{diag}\left(
        \frac{I_2-U_1J_2}{2},\ldots,
        \frac{I_2-U_\ell J_2}{2}
    \right).
\]
Each block is congruent to $I_2$ modulo $J$, so
$T'\in\GL_{2\ell}(R,J)$. Since $U_r^2=I_2$,
\[
    U_r(I_2-U_rJ_2)=(I_2-U_rJ_2)(-J_2).
\]
Consequently,
\[
    (T')^{-1}w_{[\beta_i]}T'=w_{[\beta_i]}(1)
    \qquad(1\leq i\leq\ell-1).
\]
The matrix $T'$ also commutes with every $h_{[\alpha]}(-1)$,
since these elements act as scalars on each $B_r$.


\subsection{Normalization for long roots}

For $1\leq i\leq\ell-2$, set
\[
    v_i=(T')^{-1}w_{[\alpha_i]}T'.
\]
These matrices satisfy
\[
    v_i\equiv w_{[\alpha_i]}(1)\pmod J,
    \qquad v_i^2=h_{[\alpha_i]}(-1).
\]
Since $v_i$ commutes with $h_{[\alpha_i]}(-1)$, it preserves
$B_i\oplus B_{i+1}$ and its complementary summand.
On the complement, $v_i$ is an involution congruent to the
identity and hence acts identically.
Thus it remains to determine
$A_i=v_i|_{B_i\oplus B_{i+1}}$, for which $A_i^2=-I_4$.

Choose $j=2$ if $i=1$ and $j=i-1$ otherwise; this is possible
because $\ell\geq4$. The standard inverse relation gives
\[
    h_{[\alpha_j]}(-1)v_i h_{[\alpha_j]}(-1)^{-1}=v_i^{-1}.
\]
On $B_i\oplus B_{i+1}$, the matrix $h_{[\alpha_j]}(-1)$
is $\pm\operatorname{diag}(I_2,-I_2)$, while
$A_i^{-1}=-A_i$. Since $2$ is invertible, the two diagonal
blocks of $A_i$ vanish. Its square and reduction then give
\[
    A_i=\begin{pmatrix}
        0&P_i\\
        -P_i^{-1}&0
    \end{pmatrix},
    \qquad P_i\in\GL_2(R),
    \qquad P_i\equiv I_2\pmod J.
\]
The relation
\[
    v_i w_{[\beta_i]}(1)=w_{[\beta_{i+1}]}(1)v_i
\]
uses the already normalized short-root matrices.
Their restrictions to $B_i\oplus B_{i+1}$ are
$\operatorname{diag}(-J_2,I_2)$ and
$\operatorname{diag}(I_2,-J_2)$, respectively.
Comparing the lower-left blocks yields
$P_i^{-1}J_2=J_2P_i^{-1}$, or equivalently,
\[
    P_iJ_2=J_2P_i.
\]

Define
\[
    T_1=I_2,
    \qquad T_{i+1}=P_i^{-1}T_i
    \quad(1\leq i\leq\ell-2),
    \qquad T_\ell=I_2,
\]
and put $T''=\operatorname{diag}(T_1,\ldots,T_\ell)$.
Every $T_r$ is congruent to $I_2$ modulo $J$ and commutes
with $J_2$. Hence $T''\in\GL_{2\ell}(R,J)$, and $T''$
commutes with every $h_{[\alpha]}(-1)$ and every
$w_{[\beta_i]}(1)$.
For the long simple roots, the recursion gives
\[
    T_i^{-1}P_iT_{i+1}=I_2,
\]
and therefore
\[
    \left((T'')^{-1}v_iT''\right)|_{B_i\oplus B_{i+1}}
    =\begin{pmatrix}0&I_2\\-I_2&0\end{pmatrix}.
\]
The conjugated matrix is the identity on the complement, so
\[
    (T'')^{-1}v_iT''=w_{[\alpha_i]}(1)
    \qquad(1\leq i\leq\ell-2).
\]

Every standard Weyl matrix $w_{[\alpha]}(1)$ is a word in
the simple Weyl matrices and their inverses: Weyl conjugation
gives the remaining roots, and
$w_{[\alpha]}(-1)=w_{[\alpha]}(1)^{-1}$ accounts for the
possible sign. Thus, setting
\[
    g_3=T'T''\in\GL_{2\ell}(R,J),
\]
we obtain the exact identities
\[
    g_3^{-1}w_{[\alpha]}g_3=w_{[\alpha]}(1)
    \qquad([\alpha]\in\Phi_\rho).
\]
Consequently,
\[
    \varphi_3=i_{[g_3]}^{-1}\circ\varphi_2
\]
fixes every $w_{[\alpha]}(1)$, retains the normalization of
all $h_{[\alpha]}(-1)$, and still induces $f$ modulo $J$.
This proves Step~2 of Section~\ref{sec:outline_of_main_thm}.


\section{The Images of \texorpdfstring{$x_{[\alpha]}(1)$}{x(1)}}
\label{sec:image_of_x(1)}


\subsection{Matrix representatives}
\label{subsec:choice_of_x}

For each $[\alpha]\in\Phi_\rho$, use the representative
\[
    x_{[\alpha]}
    =(g_2g_3)^{-1}\widehat\varphi_1
    \bigl(x_{[\alpha]}(1)\bigr)(g_2g_3)
\]
of $\varphi_3(x_{[\alpha]}(1))$, as in
Lemma~\ref{lemma:compatible-representatives}.
Then $x_{[\alpha]}\equiv x_{[\alpha]}(1)\pmod J$.
The map defining these representatives is a homomorphism
and fixes all standard matrices $h_{[\alpha]}(-1)$ and
$w_{[\alpha]}(1)$ exactly.
Thus all relations used below hold in matrices, without
scalar adjustments.

Weyl conjugation carries a root element with parameter $1$
to a root element of the same length with parameter $\pm1$.
It therefore suffices to normalize $x_{[\alpha_1]}$ and
$x_{[\alpha_{\ell-1}]}$ by changes of basis commuting with
the fixed Weyl matrices.
We retain the ordered decomposition
$R^{2\ell}=B_1\oplus\cdots\oplus B_\ell$ from
Subsection~\ref{subsec:SO(R)} and write
$J_2=\begin{pmatrix}0&1\\1&0\end{pmatrix}$.


\subsection{Normalization for long roots}

Put $x=x_{[\alpha_1]}$, $w_i=w_{[\alpha_i]}(1)$ for
$i=1,2$, and $h=h_{[\alpha_2]}(-1)$.
The standard Weyl conjugation identities give
\begin{align*}
    x_2&=(w_1w_2)x(w_1w_2)^{-1}=x_{[\alpha_2]},\\
    z&=w_2xw_2^{-1}=x_{[\alpha_1]+[\alpha_2]}^{-1},\\
    v&=w_2x_2w_2^{-1}=x_{-[\alpha_2]}^{-1}.
\end{align*}
In particular,
\[
    xz=zx,\qquad xv=vx,\qquad hxh^{-1}=x^{-1},
    \qquad [x,x_2]=z^{-1}.
\]
These relations will both determine the support of $x$
and normalize its nontrivial block.

Since $x$ commutes with $h_{[\alpha_1]}(-1)$, it preserves
$B_1\oplus B_2$ and its complement.
With respect to
\[
    R^{2\ell}=B_1\oplus B_2\oplus B_3\oplus B',
    \qquad B'=B_4\oplus\cdots\oplus B_\ell,
\]
write
\[
    x=\begin{pmatrix}
        A&B&0&0\\
        C&D&0&0\\
        0&0&E&F\\
        0&0&G&H
    \end{pmatrix}.
\]
Here $A,D,E\equiv I_2\pmod J$,
$H\equiv I_{2\ell-6}\pmod J$, and
\[
    B\equiv\begin{pmatrix}1&0\\0&0\end{pmatrix}\pmod J,
    \qquad
    C\equiv\begin{pmatrix}0&0\\0&-1\end{pmatrix}\pmod J.
\]
In particular, $B-C$ is invertible.
The Weyl conjugates above have the block forms
\[
    z=\begin{pmatrix}
        A&0&-B&0\\
        0&E&0&F\\
        -C&0&D&0\\
        0&G&0&H
    \end{pmatrix},
    \qquad
    v=\begin{pmatrix}
        E&0&0&F\\
        0&D&-C&0\\
        0&-B&A&0\\
        G&0&0&H
    \end{pmatrix}.
\]
The $(1,4)$- and $(4,1)$-blocks of $xz=zx$, together with
the $(2,4)$- and $(4,2)$-blocks of $xv=vx$, give
\[
    2BF=2CF=0,\qquad 2GB=2GC=0.
\]
Since $2$ and $B-C$ are invertible, $F=G=0$.
On $B_3\oplus B'$, the inverse relation $hxh^{-1}=x^{-1}$
now yields $E^2=I_2$ and $H^2=I_{2\ell-6}$.
An involution congruent to the identity modulo $J$ is the
identity, so
\[
    x=\operatorname{diag}\left(
        \begin{pmatrix}A&B\\C&D\end{pmatrix},
        I_{2\ell-4}
    \right).
\]
This argument applies uniformly for every $\ell\geq4$.

Returning to $xz=zx$ and $xv=vx$, their $(2,3)$- and
$(1,3)$-blocks, respectively, give $CB=BC=0$.
The inverse relation then gives
\[
    A^2-BC=I_2,\qquad D^2-CB=I_2,
\]
so $A=D=I_2$.
On $B_1\oplus B_2\oplus B_3$, the commutator is therefore
\[
    [x,x_2]
    =\begin{pmatrix}
        I_2&0&B^2\\
        0&I_2&0\\
        -C^2&0&I_2
    \end{pmatrix}.
\]
Comparison with $z^{-1}$ yields
\[
    B^2=B,\qquad C^2=-C.
\]
Consequently, $(B-C)^2=B-C$.
Since $B-C$ is invertible, $B-C=I_2$.

The roots $[\alpha_1]$ and $[\widetilde\alpha_1]=\nu_1+\nu_2$
are strictly orthogonal, so $x$ commutes with
$w_{[\widetilde\alpha_1]}(1)$.
On $B_1\oplus B_2$, this Weyl matrix is
$\begin{pmatrix}0&J_2\\-J_2&0\end{pmatrix}$.
Hence $C=-J_2BJ_2$, and combining this with $C=B-I_2$ gives
\[
    B=\begin{pmatrix}p&q\\-q&1-p\end{pmatrix},
    \qquad
    C=\begin{pmatrix}p-1&q\\-q&-p\end{pmatrix},
\]
where
\[
    p\equiv1\pmod J,\qquad q\in J,
    \qquad p^2-q^2=p.
\]
In particular, $p$ is a unit and $1-p=-q^2/p$.

Set
\[
    T_0=\begin{pmatrix}1&-q/p\\-q/p&1\end{pmatrix},
    \qquad
    T'=\operatorname{diag}(T_0,\ldots,T_0,I_2).
\]
Then $T'\equiv I_{2\ell}\pmod J$ and
$\det T_0=1-q^2/p^2=p^{-1}$.
The identities
\[
    BT_0=T_0\begin{pmatrix}1&0\\0&0\end{pmatrix},
    \qquad
    CT_0=T_0\begin{pmatrix}0&0\\0&-1\end{pmatrix}
\]
show that $(T')^{-1}xT'=x_{[\alpha_1]}(1)$.
Moreover, $T_0$ commutes with $J_2$, and the same block
$T_0$ occurs on $B_1,\ldots,B_{\ell-1}$.
Thus $T'$ commutes with all standard Weyl matrices and
all $h_{[\alpha]}(-1)$.
Weyl conjugation now gives
\[
    (T')^{-1}x_{[\alpha]}T'=x_{[\alpha]}(1)
    \qquad([\alpha]\sim A_1).
\]


\subsection{Normalization for short roots}

Put
\[
    y_{[\alpha]}=(T')^{-1}x_{[\alpha]}T',
    \qquad y=y_{[\alpha_{\ell-1}]}.
\]
These representatives still satisfy all transported matrix
relations, and $y_{[\alpha]}=x_{[\alpha]}(1)$ for every long
root. We first prove that
\[
    y=\operatorname{diag}(I_{2\ell-4},A),
    \qquad
    A\equiv\begin{pmatrix}
        1&-1&1&1\\
        0&1&0&0\\
        0&-1&1&0\\
        0&-1&0&1
    \end{pmatrix}\pmod J,
\]
where $A$ acts on $B_{\ell-1}\oplus B_\ell$.

For $\ell\geq5$, commutation with
$h_{[\alpha_i]}(-1)$, $1\leq i\leq\ell-3$, gives
\[
    y=\operatorname{diag}(u_1,\ldots,u_{\ell-2},A),
    \qquad u_i\equiv I_2\pmod J.
\]
The inverse relation
\[
    h_{[\alpha_{\ell-2}]}(-1)y
    h_{[\alpha_{\ell-2}]}(-1)^{-1}=y^{-1}
\]
implies $u_i^2=I_2$, so every $u_i=I_2$.

For $\ell=4$, commutation with $h_{[\alpha_1]}(-1)$ gives
$y=\operatorname{diag}(U,A)$ on
$(B_1\oplus B_2)\oplus(B_3\oplus B_4)$, with
$U\equiv I_4\pmod J$.
The matrix $U$ commutes with the restrictions of the four
already normalized long-root elements
$x_{\pm[\alpha_1]}(1)$ and
$x_{\pm[\widetilde\alpha_1]}(1)$.
Commuting with the first pair gives
\[
    U=\begin{pmatrix}
        c&0&0&-r\\
        0&d&-s&0\\
        0&r&c&0\\
        s&0&0&d
    \end{pmatrix},
\]
and the second pair forces $r=s=0$ and $c=d$.
Thus $U=cI_4$ with $c\equiv1\pmod J$.
The inverse relation with $h_{[\alpha_2]}(-1)$ gives
$c^2=1$, hence $U=I_4$.
This proves the asserted support of $y$ in every rank.

Since $y$ commutes with the normalized element
$x_{-[\alpha_{\ell-2}]}(1)$, its $4\times4$ block has the form
\[
    A=\begin{pmatrix}
        1&a&b&c\\
        0&1&0&0\\
        0&d&u&v\\
        0&e&w&z
    \end{pmatrix}.
\]
Indeed, comparing the blocks involving $B_{\ell-2}$ fixes
the first column and the second row of $A$.
The inverse relation restricts to
\[
    A\begin{pmatrix}-I_2&0\\0&I_2\end{pmatrix}A
    =\begin{pmatrix}-I_2&0\\0&I_2\end{pmatrix}.
\]
Its lower-right block says that
$\begin{pmatrix}u&v\\w&z\end{pmatrix}$ is an involution
congruent to $I_2$ modulo $J$, so it equals $I_2$.
The $(1,2)$-entry then gives $2a=bd+ce$.
Consequently,
\[
    A=\begin{pmatrix}
        1&a&b&c\\
        0&1&0&0\\
        0&d&1&0\\
        0&e&0&1
    \end{pmatrix}.
\]

The standard relation
$(x_{[\alpha_{\ell-1}]}(1)w_{[\alpha_{\ell-1}]}(1))^3=I$
transfers exactly to $y$.
On $B_{\ell-1}\oplus B_\ell$, write
$W=\operatorname{diag}(-J_2,-J_2)$.
Comparing the $(2,2)$-, $(2,3)$-, $(3,2)$-, and $(3,3)$-entries
of $(AW)^3=I_4$, successively, gives
\[
    a=-1,\qquad c=b,\qquad e=d,\qquad bd=-1.
\]
Writing $t=b\equiv1\pmod J$, we obtain
\[
    A=\begin{pmatrix}
        1&-1&t&t\\
        0&1&0&0\\
        0&-t^{-1}&1&0\\
        0&-t^{-1}&0&1
    \end{pmatrix}.
\]

Set
\[
    T''=\operatorname{diag}(I_2,\ldots,I_2,t^{-1}I_2)
    \in\GL_{2\ell}(R,J).
\]
This matrix commutes with all fixed Weyl and torus elements,
and with every $x_{[\alpha]}(1)$ for $[\alpha]\sim A_1$,
since long-root elements act identically on $B_\ell$.
Directly on the last two summands,
\[
    (T'')^{-1}yT''=x_{[\alpha_{\ell-1}]}(1).
\]
Weyl conjugation therefore normalizes all short-root elements
as well.

Finally, put
\[
    g_4=T'T''\in\GL_{2\ell}(R,J),
    \qquad \varphi_4=i_{[g_4]}^{-1}\circ\varphi_3.
\]
Then the matrix identities
\[
    g_4^{-1}x_{[\alpha]}g_4=x_{[\alpha]}(1)
    \qquad([\alpha]\in\Phi_\rho)
\]
show that $\varphi_4$ fixes all root elements with parameter
$1$, retains the normalization of all Weyl and torus elements,
and still induces $f$ modulo $J$.
This proves Step~3 of Section~\ref{sec:outline_of_main_thm}.


\section{The Images of \texorpdfstring{$x_{[\alpha]}(t)$}{x(t)}}
\label{sec:image_of_x(t)}

We complete Step~4 of Section~\ref{sec:outline_of_main_thm}.
Recall that
\[
    \varphi_5=i_{[C_2]}^{-1}\circ\varphi_1,
    \qquad C_2\in\GL_{2\ell}(R,J)^\sigma,
    \qquad [C_2]\in G_{\mathrm{ad},\sigma}(D_\ell,R).
\]
For matrix calculations, use the homomorphism
\[
    \widehat\varphi_5(x)
    =C_2^{-1}\widehat\varphi_1(x)C_2
    \qquad
    (x\in E'_{\pi_0,\sigma}(D_\ell,R)).
\]
Since $C_2$ is a $\sigma$-symmetric orthogonal similarity,
these matrices belong to $\SO_{2\ell}(R)^\sigma$.
By the construction in Section~\ref{sec:outline_of_main_thm},
$\widehat\varphi_5$ fixes every $x_{[\alpha]}(1)$, and hence
all standard Weyl and torus matrices used below.
Its reduction is the entrywise field automorphism $f$.
Thus all transported relations are exact matrix identities.

For a simple relative root, write
\[
    x_i(t)=\widehat\varphi_5\bigl(x_{[\alpha_i]}(t)\bigr)
    \qquad(t\in R_{[\alpha_i]}).
\]
We retain the ordered decomposition
$R^{2\ell}=B_1\oplus\cdots\oplus B_\ell$ from
Subsection~\ref{subsec:SO(R)}.


\subsection{The long-root parameter map}

Fix $t\in R_\theta$ and put $x=x_1(t)$.
Commuting with $h_{[\alpha_1]}(-1)$ separates
$B_1\oplus B_2$ from its complement.
Relative to
\[
    R^{2\ell}=B_1\oplus B_2\oplus B_3\oplus B',
    \qquad B'=B_4\oplus\cdots\oplus B_\ell,
\]
write
\[
    x=\begin{pmatrix}
        A&B&0&0\\
        C&D&0&0\\
        0&0&E&F\\
        0&0&G&H
    \end{pmatrix}.
\]
Reduction modulo $J$ gives
$A,D,E\equiv I_2$ and $H\equiv I_{2\ell-6}$.

The matrix $x$ commutes with the fixed elements
$x_{[\alpha_1]+[\alpha_2]}(1)$ and $x_{-[\alpha_2]}(1)$.
Put
\[
    P_1=\begin{pmatrix}1&0\\0&0\end{pmatrix},
    \qquad P_2=\begin{pmatrix}0&0\\0&1\end{pmatrix}.
\]
The first of these elements has off-diagonal blocks $P_1$
and $-P_2$ in positions $(1,3)$ and $(3,1)$, respectively;
the second has $P_1$ and $-P_2$ in positions $(3,2)$ and
$(2,3)$.
Comparing the blocks involving $B'$ in the two commutation
identities gives
\[
    P_1F=P_2F=0,
    \qquad GP_1=GP_2=0.
\]
Since $P_1+P_2=I_2$, we obtain $F=G=0$.
The inverse relation
\[
    h_{[\alpha_2]}(-1)x h_{[\alpha_2]}(-1)^{-1}=x^{-1}
\]
now gives $E^2=I_2$ and $H^2=I_{2\ell-6}$.
Both matrices are congruent to the identity, so $E=I_2$
and $H=I_{2\ell-6}$.
This establishes the required support for all $\ell\geq4$,
without any restriction on $t$.

With these blocks fixed, the same two commutation identities
reduce the remaining block to
\[
    x=\operatorname{diag}(L,I_{2\ell-4}),
    \qquad
    L=\begin{pmatrix}
        1&a&b&0\\
        0&1&0&0\\
        0&0&1&0\\
        0&c&d&1
    \end{pmatrix}.
\]
The inverse relation gives $2a=2d=0$, hence $a=d=0$.
Since $x^tSx=S$ and $\sigma(x)=x$, we further obtain
\[
    c=-b,\qquad b=\bar b.
\]
Define $\eta(t)=b\in R_\theta$. Then
\[
    x_1(t)=x_{[\alpha_1]}\bigl(\eta(t)\bigr),
    \qquad q(\eta(t))=f(q(t)).
\]
The identity $x_1(-t)=x_1(t)^{-1}$ gives
$\eta(-t)=-\eta(t)$.
Since the fixed Weyl matrices carry any long-root subgroup
to any other, with parameter changed at most by a sign,
\[
    \widehat\varphi_5\bigl(x_{[\alpha]}(t)\bigr)
    =x_{[\alpha]}\bigl(\eta(t)\bigr)
    \qquad([\alpha]\sim A_1,\ t\in R_\theta).
\]


\subsection{The short-root parameter map}

Fix $t\in R$ and put $y=x_{\ell-1}(t)$.
The support argument for short roots in
Section~\ref{sec:image_of_x(1)} applies to arbitrary $t$.
For completeness, when $\ell\geq5$, commutation with
$h_{[\alpha_i]}(-1)$ for $1\leq i\leq\ell-3$ gives
\[
    y=\operatorname{diag}(u_1,\ldots,u_{\ell-2},A),
    \qquad u_i\equiv I_2\pmod J,
\]
where $A$ acts on $B_{\ell-1}\oplus B_\ell$.
The inverse relation
\[
    h_{[\alpha_{\ell-2}]}(-1)y
    h_{[\alpha_{\ell-2}]}(-1)^{-1}=y^{-1}
\]
forces $u_i^2=I_2$, and hence $u_i=I_2$.
When $\ell=4$, commutation with $h_{[\alpha_1]}(-1)$ gives
$y=\operatorname{diag}(U,A)$.
The block $U\equiv I_4\pmod J$ commutes with the restrictions
of $x_{\pm[\alpha_1]}(1)$ and
$x_{\pm[\widetilde\alpha_1]}(1)$.
The centralizer calculation in
Section~\ref{sec:image_of_x(1)} gives $U=rI_4$;
the inverse relation gives $r^2=1$, so $r=1$.
Thus, in every rank,
\[
    y=\operatorname{diag}(I_{2\ell-4},A).
\]

Commutation with the fixed element
$x_{-[\alpha_{\ell-2}]}(1)$ determines the first column
and second row of $A$, giving
\[
    A=\begin{pmatrix}
        1&a&b&c\\
        0&1&0&0\\
        0&d&u&v\\
        0&e&w&z
    \end{pmatrix},
    \qquad
    K=\begin{pmatrix}u&v\\w&z\end{pmatrix}
    \equiv I_2\pmod J.
\]
The inverse relation restricts to
\[
    A\begin{pmatrix}-I_2&0\\0&I_2\end{pmatrix}A
    =\begin{pmatrix}-I_2&0\\0&I_2\end{pmatrix}.
\]
Its lower-right block gives $K^2=I_2$, and therefore
$K=I_2$.
Now use $y^tSy=S$.
On $B_{\ell-1}\oplus B_\ell$, the form has matrix
$\operatorname{diag}(J_2,J_2)$, where
$J_2=\begin{pmatrix}0&1\\1&0\end{pmatrix}$.
The $(2,3)$-, $(2,4)$-, and $(2,2)$-entries of the form
identity give, respectively,
\[
    e=-b,\qquad d=-c,\qquad a=-bc.
\]
Finally, $\sigma(y)=y$ gives $c=\bar b$.
Defining $\mu(t)=b$, we obtain
\[
    x_{\ell-1}(t)
    =x_{[\alpha_{\ell-1}]}\bigl(\mu(t)\bigr),
    \qquad q(\mu(t))=f(q(t)).
\]
The inverse identity gives $\mu(-t)=-\mu(t)$.

Long simple Weyl matrices permute the positive short-root
subgroups, changing their parameters at most by a sign.
For the negative roots, use
\[
    w_{[\beta_i]}(1)x_{[\beta_i]}(t)
    w_{[\beta_i]}(1)^{-1}=x_{-[\beta_i]}(-t)
    \qquad(1\leq i\leq\ell-1).
\]
All these Weyl matrices are fixed by $\widehat\varphi_5$.
Consequently,
\[
    \widehat\varphi_5\bigl(x_{[\alpha]}(t)\bigr)
    =x_{[\alpha]}\bigl(\mu(t)\bigr)
    \qquad([\alpha]\sim A_1^2,\ t\in R).
\]


\subsection{Recovery of the ring automorphism}

The short-root addition identity
\[
    x_{[\beta]}(s+t)=x_{[\beta]}(s)x_{[\beta]}(t)
\]
and injectivity of the parametrization show that $\mu$ is
additive. The normalization at parameter $1$ gives
$\eta(1)=\mu(1)=1$.

Put $[\alpha]=[\alpha_{\ell-2}]$ and
$[\beta]=[\alpha_{\ell-1}]$.
In the resulting subsystem of type $B_2$, the matrix formulas
of Subsection~\ref{subsec:SO(R)} give
\[
    [x_{[\alpha]}(u),x_{[\beta]}(v)]
    =x_{[\alpha]+[\beta]}(uv)\,
     x_{[\alpha]+2[\beta]}(uv\bar v)
    \qquad(u\in R_\theta,\ v\in R).
\]
The two parameters on the right are read from distinct
matrix entries.
Applying $\widehat\varphi_5$ and comparing the short-root
parameter yields
\[
    \mu(uv)=\eta(u)\mu(v)
    \qquad(u\in R_\theta,\ v\in R).
\]
Taking $v=1$ gives $\mu|_{R_\theta}=\eta$.
Thus
\begin{equation}\label{eq:a}
    \mu(uv)=\mu(u)\mu(v)
    \qquad(u\in R_\theta,\ v\in R).
\end{equation}

The second mixed commutator identity is
\[
    [x_{[\alpha]+[\beta]}(u),x_{-[\beta]}(v)]
    =x_{[\alpha]}(uv+\bar u\,\bar v)
    \qquad(u,v\in R).
\]
Since $uv+\bar u\,\bar v\in R_\theta$, applying
$\widehat\varphi_5$ and using $\mu|_{R_\theta}=\eta$ gives
\begin{equation}\label{eq:m}
    \mu(uv+\bar u\,\bar v)
    =\mu(u)\mu(v)
     +\overline{\mu(u)}\,\overline{\mu(v)}
    \qquad(u,v\in R).
\end{equation}
Taking $v=1$ and using additivity, we obtain
\[
    \mu(\bar u)=\overline{\mu(u)}
    \qquad(u\in R).
\]
Thus $\mu$ commutes with $\theta$ and maps
$R_\theta^-$ into itself.
For $u,v\in R_\theta^-$, equation~\eqref{eq:m} becomes
$2\mu(uv)=2\mu(u)\mu(v)$.
Since $2\in R^\times$, this gives
\begin{equation}\label{eq:b}
    \mu(uv)=\mu(u)\mu(v)
    \qquad(u,v\in R_\theta^-).
\end{equation}
By additivity and the decomposition
$R=R_\theta\oplus R_\theta^-$,
equations~\eqref{eq:a} and~\eqref{eq:b} prove
multiplicativity on all of $R$.
Hence $\mu$ is a unital ring endomorphism commuting with
$\theta$.

It remains to prove bijectivity.
The automorphism $\varphi_5^{-1}$ fixes the same standard
elements and induces $f^{-1}$ modulo $J$.
Applying Lemma~\ref{lemma:compatible-representatives} and
the preceding root-subgroup calculations to $\varphi_5^{-1}$
gives a map $\mu'\colon R\to R$ such that
\[
    \varphi_5^{-1}\bigl(x_{[\alpha]}(t)\bigr)
    =x_{[\alpha]}\bigl(\mu'(t)\bigr)
    \qquad([\alpha]\sim A_1^2,\ t\in R).
\]
The parametrization of a short-root subgroup remains
injective in the projective group, since its standard
matrices have all diagonal entries equal to $1$.
Composing the two group automorphisms and comparing
parameters therefore gives
\[
    \mu'\circ\mu=\mathrm{id}_R,
    \qquad \mu\circ\mu'=\mathrm{id}_R.
\]
Thus $\mu\in\operatorname{Aut}(R)$, with
$\mu\theta=\theta\mu$, and $\eta$ is its restriction to
$R_\theta$.

For every relative root we now have
\[
    \varphi_5\bigl(x_{[\alpha]}(t)\bigr)
    =x_{[\alpha]}\bigl(\mu(t)\bigr)
    \qquad([\alpha]\in\Phi_\rho,\ t\in R_{[\alpha]}).
\]
Since these root subgroups generate
$E'_{\mathrm{ad},\sigma}(D_\ell,R)$, this proves Step~4.
Returning to the original automorphism gives
\[
    \varphi=i_{g_1[C_2]}\circ\mu,
    \qquad g_1[C_2]\in G_{\mathrm{ad},\sigma}(D_\ell,R).
\]
This completes the proof of Theorem~\ref{MT_local1}.

\medskip

We finish by recalling the reductions from
\cite[Section~9]{EB&DM1}.
Put $E_\pi=E'_{\pi,\sigma}(D_\ell,R)$ and
$G_\pi=G_{\pi,\sigma}(D_\ell,R)$.
For Theorem~\ref{MT_local2}, pass to
\[
    E_\pi/Z(E_\pi)\cong E'_{\mathrm{ad},\sigma}(D_\ell,R)
\]
and apply Theorem~\ref{MT_local1}.
The torus--elementary factorization and character-extension
construction of \cite[Section~9]{EB&DM1} lift the inner factor
to a composition of strictly inner and diagonal automorphisms
of $G_\pi$; the ring factor lifts naturally.
After removing these factors, the remaining automorphism
$\psi$ of $E_\pi$ induces the identity on $E_\pi/Z(E_\pi)$.
Thus $e\mapsto e^{-1}\psi(e)$ is a homomorphism
$E_\pi\to Z(E_\pi)$, which is trivial by perfectness.
This proves Theorem~\ref{MT_local2}.

For Theorem~\ref{MT_local3}, $E_\pi$ is characteristic in
$G_\pi$ by \cite[Corollary~A.2]{SG&DM1}.
Extend the standard factors of the restriction to $G_\pi$
as above and remove them.
The remaining automorphism $\psi$ fixes $E_\pi$ pointwise.
For $g\in G_\pi$ and $e\in E_\pi$, normality of $E_\pi$ gives
\[
    \psi(g)e\psi(g)^{-1}
    =\psi(geg^{-1})=geg^{-1}.
\]
Consequently,
\[
    g^{-1}\psi(g)\in C_{G_\pi}(E_\pi)=Z(G_\pi),
\]
where the centralizer equality is the one used in
\cite[Section~9]{EB&DM1}.
Hence $\psi$ is central, proving Theorem~\ref{MT_local3}.


\section*{Acknowledgements}

This paper is dedicated to the memory of our friend and colleague
Eugene Plotkin.
The authors thank Boris Kunyavskii, Shripad M. Garge,
Dipendra Prasad, and Pavel Gvozdevsky for helpful discussions
and valuable comments.

\section*{Statements and Declarations}

\noindent\textbf{Funding.}
The authors received no specific funding for this work.


\smallskip
\noindent\textbf{Data availability.}
Data sharing is not applicable to this article, which is a
purely theoretical study and does not rely on research datasets.

\section*{Use of AI-assisted tools}

During the revision of this paper, ChatGPT was used to assist with the exposition, explore simplifications of proofs, and check selected calculations.
The authors take full responsibility for the mathematical content and the final manuscript.



\begin{thebibliography}{99}

\bibitem{EA0}
E.~Abe, \emph{Automorphisms of Chevalley groups over commutative rings},
St. Petersburg Math. J. \textbf{5} (1994), no.~2, 287--300;
translated from Algebra i Analiz \textbf{5} (1993), no.~2, 74--90.

\bibitem{EA1}
E.~Abe, \emph{Coverings of twisted Chevalley groups over commutative rings},
Sci. Rep. Tokyo Kyoiku Daigaku Sect. A \textbf{13} (1977), 194--218.

\bibitem{EA2}
E.~Abe, \emph{Chevalley groups over local rings},
Tohoku Math. J. (2) \textbf{21} (1969), 474--494.

\bibitem{EA3}
E.~Abe, \emph{Chevalley groups over commutative rings},
in \emph{Proceedings of the Conference on Radical Theory}
(Sendai, 1988), Uchida Rokakuho, Tokyo, 1989, 1--23.

\bibitem{EA4}
E.~Abe, \emph{Chevalley groups over commutative rings.
Normal subgroups and automorphisms},
in \emph{Second International Conference on Algebra}
(Barnaul, 1991), Contemp. Math. \textbf{184} (1995), 13--23.

\bibitem{EA&KS}
E.~Abe and K.~Suzuki,
\emph{On normal subgroups of Chevalley groups over commutative rings},
Tohoku Math. J. (2) \textbf{28} (1976), 185--198.

\bibitem{EB24:final}
E.~I.~Bunina, \emph{Automorphisms of Chevalley groups over commutative rings},
Comm. Algebra \textbf{52} (2024), 2313--2327.

\bibitem{EB12:main}
E.~I.~Bunina, \emph{Automorphisms of Chevalley groups of different types
over commutative rings},
J. Algebra \textbf{355} (2012), 154--170.

\bibitem{EB07:first}
E.~I.~Bunina, \emph{Automorphisms of Chevalley groups of certain types
over local rings},
Uspekhi Mat. Nauk \textbf{62} (2007), 143--144.

\bibitem{EB08:b2andg2}
E.~I.~Bunina, \emph{Automorphisms of Chevalley groups of types
$B_2$ and $G_2$ over local rings},
J. Math. Sci. \textbf{155} (2008), no.~6, 795--814.

\bibitem{EB09:alwithhalfele}
E.~I.~Bunina, \emph{Automorphisms of elementary adjoint Chevalley groups
of types $A_\ell$, $D_\ell$, $E_\ell$ over local rings with $1/2$},
Algebra and Logic \textbf{48} (2009), no.~4, 250--267.

\bibitem{EB10:alwithhalf}
E.~I.~Bunina, \emph{Automorphisms of Chevalley groups of types
$A_\ell$, $D_\ell$, or $E_\ell$ over local rings with $1/2$},
J. Math. Sci. \textbf{167} (2010), no.~6, 749--766.

\bibitem{EB10:alwithouthalf}
E.~I.~Bunina, \emph{Automorphisms of Chevalley groups of types
$A_\ell$, $D_\ell$, $E_\ell$ over local rings without $1/2$},
J. Math. Sci. \textbf{169} (2010), no.~5, 589--613.

\bibitem{EB10:blwithhalf}
E.~I.~Bunina, \emph{Automorphisms of Chevalley groups of type
$B_\ell$ over local rings with $1/2$},
J. Math. Sci. \textbf{169} (2010), no.~5, 557--588.

\bibitem{EB10:f4withhalf}
E.~I.~Bunina, \emph{Automorphisms of Chevalley groups of type
$F_4$ over local rings with $1/2$},
J. Algebra \textbf{323} (2010), no.~8, 2270--2289.

\bibitem{EB&DM1}
E.~I.~Bunina and D.~H.~Makadiya,
\emph{Automorphisms of twisted Chevalley groups of type
${}^2A_\ell$ $(\ell\geq5)$ over local rings},
\href{https://arxiv.org/abs/2607.22284}{arXiv:2607.22284} (2026).

\bibitem{EB&DM2}
E.~I.~Bunina and D.~H.~Makadiya,
\emph{Automorphisms of twisted Chevalley groups of type
${}^2A_3$ and ${}^2A_4$ over local rings},
\href{https://arxiv.org/abs/2609.02617}{arXiv:2609.02617} (2026).

\bibitem{EB&PV14:g2withouthlfnor}
E.~I.~Bunina and P.~A.~Ver\"evkin,
\emph{Normalizers of Chevalley groups of type $G_2$
over local rings without $1/2$},
J. Math. Sci. \textbf{201} (2014), no.~4, 446--449.

\bibitem{EB&PV14:g2withouthalf}
E.~I.~Bunina and P.~A.~Ver\"evkin,
\emph{Automorphisms of Chevalley groups of type $G_2$
over local rings without $1/2$},
J. Math. Sci. \textbf{197} (2014), no.~4, 479--491.

\bibitem{EB&MV24:g2with1/3}
E.~I.~Bunina and M.~A.~Vladykina,
\emph{Automorphisms of Chevalley groups of type $G_2$
over a commutative ring $R$ with $1/3$ generated by the invertible
elements and $2R$},
J. Math. Sci. \textbf{284} (2024), no.~4, 431--441.

\bibitem{SG&DM1}
S.~M.~Garge and D.~H.~Makadiya,
\emph{On normal subgroups of twisted Chevalley groups
over commutative rings},
J. Algebra \textbf{688} (2026), 587--659.

\bibitem{SG&DM2}
S.~M.~Garge and D.~H.~Makadiya,
\emph{Normalizer of twisted Chevalley groups over commutative rings},
J. Pure Appl. Algebra \textbf{229} (2025), no.~11,
Paper No.~108094, 14~pp.

\bibitem{SG&DM3}
S.~M.~Garge and D.~H.~Makadiya,
\emph{Triangular and unitriangular factorization of twisted Chevalley groups},
accepted for publication in Israel J. Math.;
\href{https://arxiv.org/abs/2505.20224}{arXiv:2505.20224} (2025).

\bibitem{PG:isotropic}
P.~Gvozdevsky,
\emph{Abstract isomorphisms of isotropic root graded groups
over rings},
arXiv:2505.04749 (2025).
\url{https://doi.org/10.48550/arXiv.2505.04749}

\bibitem{JH0}
J.~E.~Humphreys, \emph{On the automorphisms of infinite Chevalley groups},
Canad. J. Math. \textbf{21} (1969), 908--911.

\bibitem{AK}
A.~A.~Klyachko, \emph{Automorphisms and isomorphisms of Chevalley groups
and algebras},
J. Algebra \textbf{324} (2010), 2608--2619.

\bibitem{RS}
R.~Steinberg, \emph{Lectures on Chevalley groups},
Yale University, Department of Mathematics, New Haven, CT, 1968.

\bibitem{RS1}
R.~Steinberg, \emph{Automorphisms of finite linear groups},
Canad. J. Math. \textbf{12} (1960), 606--615.

\bibitem{KS1}
K.~Suzuki,
\emph{On normal subgroups of twisted Chevalley groups over local rings},
Sci. Rep. Tokyo Kyoiku Daigaku Sect. A \textbf{13} (1977), 238--249.

\bibitem{KS2}
K.~Suzuki, \emph{Normality of the elementary subgroups of twisted
Chevalley groups over commutative rings},
J. Algebra \textbf{175} (1995), 526--536.

\bibitem{KS3}
K.~Suzuki, \emph{Centers of twisted Chevalley groups over commutative rings},
Kumamoto J. Math. \textbf{6} (1993), 1--9.

\bibitem{GT}
G.~Taddei, \emph{Normalit\'e des groupes \'el\'ementaires dans les groupes
de Chevalley sur un anneau},
Contemp. Math. \textbf{55} (1986), part~II, 693--710.

\bibitem{LV}
L.~N.~Vaserstein,
\emph{On normal subgroups of Chevalley groups over commutative rings},
Tohoku Math. J. (2) \textbf{38} (1986), 219--230.

\bibitem{NV1}
N.~A.~Vavilov, \emph{Structure of Chevalley groups over commutative rings},
in \emph{Nonassociative algebras and related topics} (Hiroshima, 1990),
World Scientific, River Edge, NJ, 1991, 219--335.

\end{thebibliography}
\end{document}